\documentclass[11pt]{article}%
\usepackage{amsfonts}
\usepackage{amssymb,amsmath,amsthm, amsfonts}
\usepackage{mathtools}
\usepackage{color}

\makeatletter

\newcommand{\Rmnum}[1]{\expandafter\@slowromancap\romannumeral #1@}
\makeatother

\providecommand{\U}[1]{\protect \rule{.1in}{.1in}}
\newtheorem{theorem}{Theorem}[section]

\newtheorem{corollary}{Corollary}[section]

\newtheorem{definition}{Definition}[section]

\newtheorem{lemma}{Lemma}[section]

\newtheorem{proposition}{Proposition}[section]
\newtheorem{remark}{Remark}[section]

\def\argsup{\mathop{\rm argsup}}
\def\arginf{\mathop{\rm arginf}}
\def\sup{\mathop{\rm sup}}
\makeatletter
   
   \@addtoreset{equation}{section}
\makeatother

\begin{document}

\title{Value existence for zero-sum ergodic stochastic differential games\footnotemark[1]
}

\author{
Juan Li$^{1,2}$, Wenqiang Li$^{1}$, Yanwei Li$^{1}$, Huaizhong Zhao$^{2,3}$\\
{\small$^{1}$ School of Mathematics
and Statistics, Shandong University, Weihai,}\\
{\small Weihai 264209, P. R. China.}\\
{\small$^{2}$ Research Center for Mathematics and Interdisciplinary Sciences, Shandong University,}\\
{\small Qingdao 266237, P.R. China.}\\
{\small$^{3}$ Department of Mathematical Sciences, Durham University, DH1 3LE, UK.}\\
{\small \textit{E-mails: juanli@sdu.edu.cn; wenqiangli@sdu.edu.cn; 201916496@mail.sdu.edu.cn; huaizhong.zhao@durham.ac.uk}}\\
}
\footnotetext[1]{
Juan Li is supported by the NSF of P.R. China (No. 12031009, 11871037), National Key R and D Program of China (No. 2018YFA0703900), and also supported by NSFC-RS (No. 11661130148; NA150344). Wenqiang Li is supported by the China Scholarship Council, the NSF of P.R. China (No. 12101537, 12271304) and the NSF of Shandong Province (No. ZR2025MS45). Huaizhong Zhao is supported by the EPSRC (ref. EP/S005293/2).
}
\footnotetext[2]{Corresponding authors: Wenqiang Li and Yanwei Li.}

\date{\today}
\maketitle


\bigskip
\noindent \textbf{Abstract.}
In this paper we investigate two-player zero-sum stochastic differential games with an ergodic payoff, in which the diffusion coefficient does not need to be non-degenerate. We first establish  the existence of a viscosity solution to the associated ergodic Hamilton-Jacobi-Bellman-Isaacs equation under a dissipativity condition. With the help of this viscosity solution, we then derive estimates for the upper and the lower ergodic value functions by constructing a series of non-degenerate approximating processes combined with the sup- and inf-convolution techniques. Finally, we prove the existence of a value for the game under the Isaacs condition and provide its representation formulae. As an application, we study the pollution accumulation problem with a long-run average social welfare to illustrate our theoretical results.
\bigskip


\section{{\protect \large {Introduction}}}

Stochastic differential games (SDGs) have long been a central topic in applied mathematics and control theory, offering a robust framework for modeling strategic interactions in dynamic systems subject to random disturbances. In particular, two-player zero-sum SDGs with ergodic (long-run average) payoff criteria are of significant interest in economics, finance, and engineering, where the objective is to optimize the average performance over an infinite time horizon. Such problems naturally lead to the study of associated ergodic Hamilton-Jacobi-Bellman-Isaacs (HJBI) equations, which are fully nonlinear second-order partial differential equations (PDEs) whose solutions characterize the value of the game and optimal strategies. 

We consider a type of two-player zero-sum SDGs, of which the dynamics is driven by a doubly controlled stochastic differential equation (SDE)
 \begin{equation}\label{dynamics}
\left\{
\begin{array}{l}
dX_s^{x,u,v}=b(X_s^{x,u,v},u_s,v_s)ds+\sigma(X_s^{x,u,v},u_s,v_s)dB_s,\ s\geq 0,\\
X_0^{x,u,v}=x,\ x\in\mathbb{R}^n,
\end{array}
\right.
\end{equation}
where $u$ and $v$ represent admissible controls for Player 1 and Player 2, respectively, $(0,x)$ is the initial time and position.
The payoff functional of our SDGs is an ergodic one, namely the large time $T$ limit of average payoff of the following form
\begin{equation}\label{equ 2.2}
J(T,x,u,v)=\frac{1}{T}E[\int_0^Tf(X_s^{x,u,v},u_s,v_s)ds].
\end{equation}
The assumptions on the coefficients $b,\sigma$ and $f$ will be given in Section 2. The objective is to study the existence of the value of  such a game when $T$ tends to infinity, i.e. the consistency of the upper ergodic value function $\rho^+(x)$ and the lower ergodic value function $\rho^-(x)$ (see \eqref{equ 2.3} for the definitions of $\rho^+$ and $\rho^-$). Moreover, both of these two functions should be independent of the initial position $x$.

The study of zero-sum SDGs with two-players, with the dynamics given by controlled SDE \eqref{dynamics}, and the payoff functional $J(T,x,u,v)$ for some fixed terminal time $T$, has been rapidly developed since the pioneering work of Fleming and Souganidis \cite{FS1989}. They proved that the lower and upper value functions satisfy the dynamic programming principle (DPP) and are the unique viscosity solutions of the associated Bellman-Isaacs equations. Under the so called Isaacs condition, these two value functions coincide which implies that the value exists. Many results on SDGs were obtained such as that with recursive payoffs by Buckdahn, Li \cite{BL2008} and Li, Wei \cite{LW2015}, with jumps by Biswas \cite{B2012} and Buckdahn, Hu, Li \cite{BHL2011}, with reflection by Buckdahn, Li  \cite{BL2011}, with asymmetric information by Cardaliaguet, Rainer \cite{CR2009},  without the Isaacs condition by Buckdahn, Li, Quincampoix \cite{BLQ20132} and Li, Li \cite{LL2017}, etc. It is noted that all these works focused on finite time horizon SDGs.

For infinite horizon SDGs, two payoff criteria, namely the discounted payoff criterion and the ergodic one, are usually considered. In particular, Borkar, Ghosh \cite{BG1992} studied both of these two criteria based on occupation measures and considered relaxed feedback controls  to ensure the existence of the invariant measure of state processes. Arapostathis, Borkar, Kumar \cite{ABK2013} introduced a relative value iteration scheme to study game problems similar to \cite{BG1992}. Notice that in both \cite{BG1992} and \cite{ABK2013}, the diffusion coefficients need to be non-degenerate and independent of the controls.
Alvarez, Bardi \cite{AB2007} proved the existence of the values of SDGs with the ergodic payoff via the long time limit of value functions of the corresponding finite horizon games that are characterised in terms of the related Hamilton-Jacobi-Isaacs equation.  However, in their work, the diffusion coefficient of the dynamics also needs to be non-degenerate, and the coefficients of the dynamics as well as the performance function of ergodic payoff need to be periodic with respect to the state variables.

How to study infinite time horizon SDGs without a non-degenerate assumption in the ergodic settings (i.e., the value independent of the initial condition) is one of the main challenges in the differential game theory (see Section 3,  Buckdahn, Cardaliaguet and Quincampoix  \cite{BCM2011} for more details). Very little is known without such an assumption for infinite time horizon SDGs under the ergodic settings. However, there have been some works on stochastic control problems and SDGs under non-ergodic settings, i.e., the value of the game may  depend upon the initial condition. We refer to Buckdahn, Goreac, Quincampoix \cite{BGQ2014}, Li, Zhao \cite{LZ2019}, Quincampoix, Renault \cite{QR2011}, Buckdahn, Li, Quincampoix, Renault \cite{BLQR2020} for control problems  and Buckdahn, Li, Zhao \cite{BLZ2018} for
game problems. In order to study the existence of asymptotic values, in these works certain suitable non-expansivity conditions, instead of the non-degeneracy condition, are given. In contrast to these results,
our work focuses on the ergodic setting for SDGs, while the non-degeneracy condition is no longer required.

The classical PDEs techniques (see \cite{FS1989}) used to prove the existence of the values in finite time SDGs  may fail to work for infinite time SDGs. As a result, we use some idea of \'Swiech \cite{S1996} to study our problem. The basic framework is that we first study the existence of the solution of the related PDEs and then give the estimates for the upper and lower ergodic value functions in terms of this solution, from which we conclude the existence of the value for  SDGs under the  Isaacs condition. Compared with the traditional way introduced in \cite{FS1989}, this approach in some sense is a reverse process. It was first introduced by \'Swiech to obtain the sub- and super-optimality principles of DPP for deterministic differential games in \cite{S1996a} and SDGs in \cite{S1996} with discounted payoff functionals.
In the following, we briefly introduce the framework and explain our main contributions.

We first study a type of \emph{ergodic} HJBI equations (see \eqref{pde1} and \eqref{pde2} for their forms), corresponding to SDGs. The existence of viscosity solutions of these equations is obtained by vanishing the discounted factor limit of the associated PDE (Theorem \ref{th3.1}). As a byproduct, we show that the long-run average  solution of the corresponding finite horizon nonlinear (possibly degenerate) second order HJBI equation converges to a scalar, which is exactly the component of the solution of the associated ergodic HJBI equation.
Similar convergence result has been obtained by Alvarez, Bardi \cite{AB2003} under the  non-degenerate condition for the uncontrolled (i.e., no players) PDEs, by Cosso, Fuhrman, Pham \cite{CFP2016} under the assumption of the existence of the related invariant measure for the single controlled, i.e., only one player, PDEs. All these assumptions are not required here and our approach is different from the PDE technique used in \cite{AB2003} and the backward SDE tool exploited in \cite{CFP2016}.

As shown in \cite{AB2007,ABK2013,BG1992}, the non-degeneracy condition ensures the ergodic property of the diffusion process and plays an important role in the study of infinite horizon SDGs. In order to overcome the absence of the non-degeneracy assumption, we construct a series of non-degenerate diffusion processes, which are different from those constructed in \cite{S1996}, to approximate the state process (Lemma \ref{le1}). Notice that we consider an infinite time horizon problem,
the classical Gaussian's upper bounds for the density functions of non-degenerate (non-)Markovian processes cannot apply to our problem (see Remarks \ref{re4.1} and \ref{re-Krylov estimate} for more details). For this, we prove a weaker upper bound for density functions using a purely probabilistic approach (Proposition \ref{le2}). This plays an important role for the study of the SDGs.  Using the sup-convolution technique, we construct a series of functions smoothing enough to approximate the viscosity solution of the corresponding ergodic HJBI equations.
With the help of the weaker upper bound estimate and the approximation approach of the viscosity solution,  we give estimates for the upper ergodic value function $\rho^+(x)$ and the lower ergodic value function $\rho^-(x)$ in terms of viscosity solutions of equations $\eqref{pde1}$ and $\eqref{pde2}$, respectively (Theorem \ref{th1}). These estimates allow us to conclude the existence of the value of SDGs with ergodic payoff under the Isaacs condition (Theorem \ref{th2}) and the uniqueness result for the corresponding ergodic HJBI equation (Corollary \ref{co1}). As a byproduct, we obtain the related DPP and several representation results for the ergodic value (Theorem \ref{character1}, Theorem \ref{th3}).

Although the proof of Theorem 4.1 follows a framework similar to that of Theorem 2.1 in \cite{S1996}, it is worth  pointing out that there will be some difficulties in dealing with the ergodic payoff instead of the discounted one. The main technical difficulty is that Krylov's estimate about the distribution of a stochastic integral used in \cite{S1996} may not hold in our ergodic situation (see Remark \ref{re-Krylov estimate} for details). For this, we introduce a purely probabilistic approach (Proposition 4.1). The advantage of this approach is that it allows both players to consider  admissible controls over the original probability space rather than in an enlarged probability space that was considered in  \cite{S1996}.  We believe that this will be easier for players to choose their optimal controls in a smaller admissible control set, and it may facilitate the possibility of numerical computations.
It is worth  noting that Proposition 4.1 has its own interest, especially in the study of the upper bound of the fundamental solutions of stochastic PDEs on \emph{infinite time horizon}.
In addition, we stress that the assumptions of the coefficients we pose here are  weaker than those needed in \cite{S1996}. This certainly creates some  difficulties.
Finally, we mention that deterministic differential games with ergodic payoff have been studied by Ghosh, Rao \cite{GR2005}, in which most of the difficulties will vanish in the deterministic situation from the results of this paper.

Let us highlight the main innovations and contributions of our work:\\
1) Existence of viscosity solutions for degenerate ergodic HJBI equations: Using a vanishing discount factor approach, we establish the existence of a viscosity solution to the ergodic HJBI equation, without assuming non-degeneracy or the existence of an invariant measure.\\
2) Novel approximation techniques: To overcome the absence of non-degeneracy, we construct a family of non-degenerate approximating processes via linear SDEs driven by an extended Brownian motion. This differs fundamentally from earlier approximation schemes and allows us to maintain admissible controls adapted to the original filtration.\\
3) Probabilistic upper bounds for density functions: Since classical Gaussian-type bounds for fundamental solutions do not apply in our infinite-horizon, non-Markovian setting, we develop a purely probabilistic method to derive weaker upper bounds for the density functions of the approximating processes. This result has its own importance for studying parabolic equations with stochastic coefficients over infinite time horizons.\\
4) Convergence and representation results: By combining sup- and inf-convolution techniques with the density estimates, we derive precise estimates for the upper and lower ergodic value functions in terms of the viscosity solutions. This enables us to prove the existence of a game value under the Isaacs condition and to obtain representation formulas and a DPP for the ergodic problem.\\
5) Application to environmental economics: We illustrate the theoretical framework through a pollution accumulation model with long-run average social welfare, where the decay rate of pollution is uncertain and controls are non-Markovian. Our results provide explicit optimal policies without imposing additional non-degeneracy or stability conditions required in prior studies.

Let us also emphasise the main difficulties and technical challenges we have had to address.\\
1) Handling the infinite time horizon and ergodic payoff, which complicate the direct application of finite-horizon techniques such as Krylov's estimates for the distribution of stochastic integrals.\\
2) Managing the lack of non-degeneracy and the resulting absence of classical Gaussian bounds for transition densities.\\
3) Preserving the original probabilistic setting where controls are adapted only to the Brownian motion driving the state process, rather than an enlarged filtration, which is crucial for practical interpretability and numerical feasibility.\\
4) Establishing convergence and comparison results for viscosity solutions in a degenerate, fully nonlinear context without relying on invariant measures or strong regularity assumptions.

In summary, this paper provides a comprehensive analysis of ergodic SDGs with degenerate diffusions, introducing new analytical and probabilistic tools to address previously unresolved challenges. The results not only advance the theoretical understanding of long-run stochastic games but also broaden their applicability to realistic models in economics and environmental policy.

The paper is organized as follows: Section 2 formulates the SDGs framework and states main assumptions. Section 3 establishes the existence of viscosity solutions for ergodic HJBI equations. Section 4 proves the existence of the game value, uniqueness of the ergodic solutions, and related representation theorems. To illustrate our results, an example on a pollution accumulation problem with a long-run average social welfare is discussed in Section 5. The property of sup-convolution and inf-convolution used in the viscosity solution theory is given in the Appendix.

\section{ {\protect \large Formulation of the SDG setting}}

Let $B$ be a standard $d$-dimensional Brownian motion defined on a complete  probability space $(\Omega, \mathcal{F}, P)$, and we denote by $\mathbb{F}^B=(\mathcal{F}^B_t)_{t\geq 0}$ the complete filtration generated by  $(B_t)_{t\geq 0}$.
Let $U$ and $V$ be two compact metric spaces which represent the control state spaces for
Player 1 and Player 2, respectively. We introduce the admissible control spaces for players 1 and 2, respectively,
\begin{equation}\label{040801}
\begin{aligned}
&\mathcal{U}=\{u|u\ \text{is\ a}\ U\text{-valued\ and}\  \mathbb{F}^B\text{-adapted\ process}\};\\
&\mathcal{V}=\{v|v\ \text{is\ a}\ V\text{-valued\ and}\ \mathbb{F}^B\text{-adapted\ process}\}.
\end{aligned}
\end{equation}
We assume that the mappings
$$b:\mathbb{R}^n\times U\times V\rightarrow \mathbb{R}^n,\ \sigma:\mathbb{R}^n\times U\times V\rightarrow \mathbb{R}^{n\times d},\ f:\mathbb{R}^n\times U\times V\rightarrow \mathbb{R},$$
satisfy the following standard conditions
\begin{description}

\item[\textbf{(H1)}]\ For every fixed $x\in\mathbb{R}^n$, $b,\sigma,f$ are  continuous in $(u,v)\in U\times V$;

\item[\textbf{(H2)}]\ For $l\in\{b,f\}$ and $\sigma$, there exist constants $C_l$ and $C_\sigma$ such that, for all $x,y\in\mathbb{R}^n$, $u\in U$, $v\in V$,
$$|l(x,u,v)-l(y,u,v)|\leq C_l|x-y|,\ \|\sigma(x,u,v)-\sigma(y,u,v)\|\leq C_\sigma|x-y|.
$$
\end{description}
{Herein,  $\|\cdot\|$  stands for the trace norm of the matrix, and  $|\cdot|$ denotes the usual norm for the vector in Euclidean space.}
Under conditions (H1) and (H2), it is easy to check that there exists a constant $C>0$, such that, for all $(u,v)\in\times U\times V$,
$$|b(x,u,v)|+\|\sigma(x,u,v)\|+|f(x,u,v)|\leq C(1+|x|).$$
Then, for any given $x\in\mathbb{R}^n$ and every pair of admissible controls $(u,v)\in\mathcal{U}\times\mathcal{V}$,  SDE \eqref{dynamics} has a unique $\mathbb{R}^n$-valued, $\mathbb{F}^B$-adapted continuous solution $(X_t^{x,u,v})_{t\geq 0}$.
Moreover, we impose the following monotonicity condition for the coefficient $b$.

\begin{description}
\item[\textbf{(H3)}]\ There exists a constant $K>(C_\sigma)^2$ such that, for all $(u,v)\in U\times V$, $x,y\in\mathbb{R}^n$,
 $$2\langle x-y, b(x,u,v)-b(y,u,v)\rangle\leq -K|x-y|^2.$$
\end{description}
\begin{remark}
The assumptions {\rm (H2)} and {\rm (H3)} imply the classical dissipativity condition, i.e., for all $(u,v)\in U\times V$, $x,y\in\mathbb{R}^n$,
 $$2\langle x-y, b(x,u,v)-b(y,u,v)\rangle+\|\sigma(x,u,v)-\sigma(y,u,v)\|^2\leq -(K-C_\sigma^2)|x-y|^2.$$
It is well known that this condition will ensure the existence and uniqueness of the invariant measure of flows generated by the solution of SDE \eqref{dynamics} when both of the coefficients $b,\sigma$ are independent of the controls $u,v$ or  $u,v$ are feedback controls (see, e.g. Proposition 2.1 in \cite{CFP2016} or Theorem 6.3.2 in \cite{PZ1996}). However, the corresponding  transition probabilities of the state process $X^{x,u,v}$ are actually not time-homogenous due to the presence of the controls $(u,v)$ in our framework. Thus, neither the classical theory of ergodicity for homogeneous Markov semigroups nor recent developments of the ergodic theory of time-periodic Markov semigroups in \cite{FZ2020} may apply immediately to our problem.
 \end{remark}

 From standard estimates  for the solution of SDE \eqref{dynamics}, we can get the following result.
  \begin{lemma}\label{estimates}
Under assumptions {\rm (H1)-(H3)},
there exist  constants $C$, $c>0$ ($c
\neq K$) such that for all $t>0,\ \delta>0$, $(u,v)\in\mathcal{U}\times\mathcal{V}$, $x$, $y\in\mathbb{R}^n$, we have the following estimates,
\begin{equation}\label{051401}
\begin{aligned}
&E|X_t^{x,u,v}|^2\leq C(1+|x|^2e^{{-c t}});\\
&E|X_{t}^{x,u,v}-X_t^{y,u,v}|^2\leq e^{-c t}|x-y|^2;\\
&E[\sup_{t\leq s\leq t+\delta}|X_s^{x,u,v}-X_t^{x,u,v}|^2]\leq C(\delta^2+\delta).
\end{aligned}
 \end{equation}
 \end{lemma}
{We emphasize that the constants $C$ and $c$ in Lemma \ref{estimates} are independent of the time variable $t$ and the control pair $(u,v)$. Moreover, the first and second estimates in \eqref{051401} can be proved under assumptions (H1)-(H3) but without the Lipschitz condition of the coefficient $b$ in $x$ as given in assumption (H2) (see, e.g., subsection 2.1 in \cite{FQZ2019}).}

{Our game problem is of the ``strategy against control" type, similar to most references on differential games such as \cite{FS1989, BHL2011, BL2008}, etc. We introduce the definition of an admissible strategy.}
\begin{definition} An admissible strategy for Player 1 is a mapping $\alpha: \mathcal{V}\rightarrow\mathcal{U}$ satisfying the following non-anticipative property: For all  $t\in[0,\infty)$ and all controls $v,v'\in\mathcal{V}$, if $v=v'$, dsdP-a.e., on $[0,t]$, then $\alpha(v)=\alpha(v')$, dsdP-a.e., on $[0,t]$. The set of all admissible strategies for Player 1 is denoted by $\mathcal{A}$.

An admissible strategy $\beta: \mathcal{U}\rightarrow\mathcal{V}$ for Player 2
is similar. Denote by $\mathcal{B}$ the set of all admissible strategies for Player 2.
\end{definition}
We consider, for $(T,x)\in[0,\infty)\times\mathbb{R}^n$, $(u,v)\in \mathcal{U}\times\mathcal{V}$, the average payoff criterion over $[0,T]$  of the form \eqref{equ 2.2}.
Using assumption (H2) and Lemma \ref{estimates}, it is easy to check that the payoff $J$ is bounded and Lipschitz in $x$, uniformly with respect to $(T,u,v)\in(0,\infty)\times U\times V$.
 Let us now introduce the following upper and lower ergodic value functions:
\begin{equation}\label{equ 2.3}
\begin{split}
\rho^+(x)=\sup\limits_{\beta\in\mathcal{B}}\inf_{u\in\mathcal{U}}\limsup_{T\rightarrow\infty}
J(T,x,u,\beta(u)),\ x\in\mathbb{R}^n,\\
\rho^-(x)=\inf_{\alpha\in\mathcal{A}}\sup_{v\in\mathcal{V}}\liminf_{T\rightarrow\infty}
J(T,x,\alpha(v),v),\ x\in\mathbb{R}^n.
\end{split}
\end{equation}
\begin{remark}
It is worth pointing out that the definition \eqref{equ 2.3} of the upper and lower ergodic value functions is different from that given in the deterministic differential games with ergodic payoff \cite{GR2005}. In \cite{GR2005}, the lower ergodic value function $\rho^-(x)$ is defined as the upper limit of $J$, i.e.,
$$\rho^-(x)=\inf_{\alpha\in\mathcal{A}}\sup_{v\in\mathcal{V}}\limsup_{T\rightarrow\infty}
J(T,x,\alpha(v),v).$$
Here we make this slight modification in the definition of the function $\rho^-(x)$  in order to admit a more general framework.
\end{remark}
If  for all $x\in\mathbb{R}^n$,
$$\rho^+(x)=\rho^-(x)=\rho,$$
where $\rho$ is a constant, we say that the SDG with ergodic payoff criterion has a value (we also say that the long time average cost game is ergodic, see \cite{AB2007} for more details). Our main aim in this paper is to study the existence of the value of the SDGs associated with \eqref{dynamics} and \eqref{equ 2.2}.
For this, we consider the following related \emph{ergodic} HJBI equations,
\begin{equation}\label{pde1}
\rho=\inf_{u\in U}\sup_{v\in V}H(x,Dw(x),D^2w(x),u,v),\ x\in\mathbb{R}^n,
\end{equation}
and
\begin{equation}\label{pde2}
\rho=\sup_{v\in V}\inf_{u\in U}H(x,Dw(x),D^2w(x),u,v),\ x\in\mathbb{R}^n,
\end{equation}
where the Hamiltonian function $$H(x,p,A,u,v)=\frac{1}{2}tr\big((\sigma\sigma^*)(x,u,v)\cdot A\big)+\langle b(x,u,v), p\rangle+f(x,u,v),$$ $(x,p,A,u,v)\in \mathbb{R}^n\times\mathbb{R}^n\times \mathcal{S}{(n)}\times U\times V$, and by $\mathcal{S}(n)$ denote the set of all $n\times n$ symmetric matrices. The solution of \eqref{pde1} (resp. \eqref{pde2}) is a couple $(\rho,w)\in\mathbb{R}\times C(\mathbb{R}^n)$.

\section{ {\protect \large The existence of viscosity solutions of ergodic HJBI equations}}
In this section, we focus on the study  of the viscosity solution of the ergodic HJBI equation \eqref{pde1} using the vanishing limit method in the discounted payoff case. The relationship with the corresponding finite horizon nonlinear (possibly degenerate) second order HJBI equation is also investigated with the help of the viscosity solution of the ergodic HJBI equation \eqref{pde1}. Note that all results in this section can be obtained analogously for ergodic HJBI \eqref{pde2}.

For convenience of the reader, we first give the definition of the viscosity solution of (\ref{pde1}). The reader is also referred to \cite{CIL1992} for more details.

\begin{definition}
\emph{(i)} A viscosity subsolution of (\ref{pde1}) is a pair $(\rho,w)$, where $\rho$ is a
real number and $w$ is continuous  on $\mathbb{R}^n$, such that for $x\in\mathbb{R}^n$ and a test function $\varphi\in C_b^3(\mathbb{R}^n)$, we have
$$\rho\leq \inf_{u\in U}\sup_{v\in V}H(x,D\varphi(x),D^2\varphi(x),u,v),$$
whenever $w-\varphi$ has a local maximum at $x$.

\emph{(ii)} A viscosity supersolution of (\ref{pde1}) is a pair $(\rho,w)$, where $\rho$ is a
real number and $w$ is continuous on $\mathbb{R}^n$  such that for $x\in\mathbb{R}^n$ and a test function $\varphi\in C^3_b(\mathbb{R}^n)$, we have
$$\rho\geq \inf_{u\in U}\sup_{v\in V}H(x,D\varphi(x),D^2\varphi(x),u,v),$$
whenever $w-\varphi$ has a local minimum at $x$.

\emph{(iii)} A viscosity solution of (\ref{pde1}) is  a pair $(\rho,w)$ that is both the viscosity subsolution and supersolution of (\ref{pde1}).
\end{definition}
Herein, by $C^3_b(\mathbb{R}^n)$ we denote the set of real-valued functions that are continuously differentiable up to the third order and whose derivatives of order 1 to 3 are bounded.
\begin{theorem}\label{th3.1}
Let Assumptions \emph{(H1)-(H3)} hold. Then the ergodic HJBI equation $\eqref{pde1}$  has
a viscosity solution $(\rho,w)$ where $w$ satisfies the following property: There exists a constant $C>0$ such that for all $x,y\in\mathbb{R}^n$, it holds
\begin{equation}\label{042102}
\begin{aligned}
|w(x)-w(y)|\leq C|x-y|,\
|w(x)|\leq C|x|.
\end{aligned}
\end{equation}
\end{theorem}
\begin{proof}
For any fixed $\lambda>0$, {it follows from the classical result (see, for example, Theorem 3.4 in the recent work \cite{L2020}) that}
\begin{equation}\label{2019041501}
w_\lambda(x)=\sup_{\beta\in\mathcal{B}}\inf_{u\in\mathcal{U}}E[\int_0^\infty e^{-\lambda s}f(X_s^{x,u,\beta(u)},u_s,\beta(u)_s)ds]
\end{equation}
is the  unique viscosity solution of the following HJBI equation
\begin{equation}\label{2019041502}
\lambda w_\lambda(x)=\inf_{u\in U}\sup_{v\in V} H(x,Dw_\lambda(x), D^2w_\lambda(x),u,v),\ x\in\mathbb{R}^n.
\end{equation}
From Lemma \ref{estimates}, we know for all $u\in\mathcal{U}$, $\beta\in\mathcal{B}$, $x,y\in\mathbb{R}^n$
\begin{equation}\label{041509}
\begin{aligned}
&|E[\int_0^\infty e^{-\lambda s}f(X_s^{x,u,\beta(u)},u_s,\beta(u)_s)ds]-E[\int_0^\infty e^{-\lambda s}f(X_s^{y,u,\beta(u)},u_s,\beta(u)_s)ds]|\\
\leq &C_fE[\int_0^\infty e^{-\lambda s}|X_s^{x,u,\beta(u)}-X_s^{y,u,\beta(u)}|ds]
\leq \frac{C_f}{\lambda+\frac{c}{2}}|x-y|\leq C|x-y|,
\end{aligned}
\end{equation}
where  the constant $C$ depends only on the Lipschitz constant of $f$ and the constant $c$ in Lemma \ref{estimates}.
It follows from \eqref{2019041501} and \eqref{041509} that
\begin{equation}\label{2019041503}
\begin{aligned}
&|w_\lambda(x)-w_\lambda(y)|\\
\leq& \sup_{\beta\in\mathcal{B}}\sup_{u\in\mathcal{U}}\Big|E[\int_0^\infty e^{-\lambda s}f(X_s^{x,u,\beta(u)},u_s,\beta(u)_s)ds]-E[\int_0^\infty e^{-\lambda s}f(X_s^{y,u,\beta(u)},u_s,\beta(u)_s)ds]\Big|\\
\leq& C|x-y|.
\end{aligned}
\end{equation}
For any $\lambda>0$, we now define
$$\rho_\lambda=\lambda w_\lambda(0),\ \varphi_\lambda(x)=w_\lambda(x)-w_\lambda(0),\ x\in\mathbb{R}^n.$$
Noting that $|f(x,u,v)|\leq C(1+|x|),$ it follows from \eqref{051401} and \eqref{2019041501} that, for every fixed $x\in\mathbb{R}^n$, $\{\lambda w_\lambda(x)\}_{\lambda>0}$ is bounded by $C(1+|x|)$, uniformly with respect to $\lambda>0$. In particular,  $\{\rho_\lambda\}_{\lambda>0}$ is also bounded, which means that there exists a subsequence $\{\rho_{\lambda_k}\}_{k\geq 1}$ with $\lambda_k\rightarrow 0$ as $k\rightarrow\infty$ such that
$\{\rho_{\lambda_k}\}$ converges, and we put $$\rho=\lim_{k\rightarrow\infty}\lambda_kw_{\lambda_k}(0).$$ On the other hand, from \eqref{2019041503} we get, for all $x,y\in\mathbb{R}^n$,
\begin{equation}\label{042101}
|\varphi_\lambda(x)-\varphi_\lambda(y)|\leq C|x-y|,\ \sup_{\lambda}|\varphi_\lambda(x)|\leq C|x|.
\end{equation}
Consequently Arzel\`a-Ascoli theorem combined with some standard arguments reveal that
 there exists a subsequence $(\lambda_l)_{l\geq 1}$ with $\lambda_l\downarrow0$ and a continuous function $w$ such that
$\varphi_{\lambda_l}\rightarrow w$, as $l\rightarrow \infty$, uniformly on any compact subset of $\mathbb{R}^n$. Obviously, $\eqref{042101}$ implies \eqref{042102} for $w$.

We now show that $(\rho,w)$ is a viscosity solution of $\eqref{pde1}$.
For $(x,\eta,p,A)\in\mathbb{R}^n\times\mathbb{R}\times\mathbb{R}^n\times\mathcal{S}(n)$, we define
\begin{equation}\nonumber
\begin{aligned}
&\displaystyle G_k(x,\eta,p,A)=\rho_{\lambda_k}+\lambda_k\eta-\inf_{u\in U}\sup_{v\in V} H(x,p,A,u,v),\ k\in N^+,\\
&\displaystyle G(x,\eta,p,A)=\rho-\inf_{u\in U}\sup_{v\in V} H(x,p,A,u,v).
\end{aligned}
\end{equation}
Then equation \eqref{2019041502} can be rewritten as
$$G_k(x,\varphi_{\lambda_k}(x),D\varphi_{\lambda_k}(x),D^2\varphi_{\lambda_k}(x))=0.$$
From the fact that $G_k\rightarrow G$ uniformly on compacts, as $k\rightarrow\infty$ and the stability property of viscosity solutions (see Remark 6.3 in \cite{CIL1992}), we know that $(\rho,w)$ is a viscosity solution of $\eqref{pde1}$.
\end{proof}
For
$T>0$, we turn our attention to the long time behaviour of the solution of the following second order HJBI equation
\begin{equation}\label{050101}
\left\{
\begin{array}{l}
\frac{\partial}{\partial t}V(t,x)=\inf_{u\in U}\sup_{v\in V}H(x,DV(t,x),D^2V(t,x),u,v),\  (t,x)\in[0,T)\times\mathbb{R}^n,\\
V(0,x)=\Phi(x),\ x\in \mathbb{R}^n,
\end{array}
\right.
\end{equation}
as $T\rightarrow\infty$. Here, $\Phi$ is a Lipschitz function on $\mathbb{R}^n$.
Due to Theorems 4.2 and 5.3 in \cite{BL2008},  the following PDE
\begin{equation}\label{041504}
\left\{
\begin{array}{l}
\frac{\partial}{\partial t}\bar{V}(t,x)+\inf_{u\in U}\sup_{v\in V}H(x,D\bar{V}(t,x),D^2\bar{V}(t,x),u,v)=0,\ (t,x)\in[0,T)\times\mathbb{R}^n,\\
\bar{V}(T,x)=\Phi(x),\ x\in\mathbb{R}^n,
\end{array}
\right.
\end{equation}
has a unique viscosity solution which can be represented as
\begin{equation}\nonumber
\bar{V}(t,x)=\sup_{\beta\in\mathcal{B}}\inf_{u\in\mathcal{U}}E[\int_t^T f({X_s^{t,x,u,\beta(u)}},u_s,\beta(u)_s)ds
+\Phi({X_T^{t,x,u,\beta(u)}})],
\end{equation}
{where $X^{t,x,u,\beta(u)}$ is the unique $\mathbb{F}$-adapted solution of SDE \eqref{dynamics} with the initial data $(t,x)$ and the parameter $(u,\beta)$. Obviously, $X^{0,x,u,\beta(u)}=X^{x,u,\beta(u)}$.}
Notice that \eqref{050101} is obtained from \eqref{041504} by ``time reversal", i.e., $V(t,x)=\bar{V}(T-t,x)$, $(t,x)\in[0,T]\times\mathbb{R}^n$. Then $V$ is the unique viscosity solution of \eqref{050101}
and
\begin{equation}\label{041507}
V(t,x)=\sup_{\beta\in\mathcal{B}}\inf_{u\in\mathcal{U}}E[\int_{T-t}^T f({X_s^{T-t,x,u,\beta(u)}},u_s,\beta(u)_s)ds
+\Phi({X_T^{T-t,x,u,\beta(u)}})].
\end{equation}
In particular,
\begin{equation}\label{041508}
V(T,x)=\sup_{\beta\in\mathcal{B}}\inf_{u\in\mathcal{U}}E[\int_{0}^T f(X_s^{x,u,\beta(u)},u_s,\beta(u)_s)ds+\Phi(X_T^{x,u,\beta(u)})].
\end{equation}

\begin{theorem}\label{character1}
Suppose $(\rho,w)$ is a viscosity solution of the ergodic HJBI equation \eqref{pde1} {satisfying \eqref{042102}} and $V$ is the viscosity solution of \eqref{050101} with a form \eqref{041507}. Then there exists a positive constant $C$ (independent of $T$) such that
$$|V(T,x)-(\rho T+w(x))|\leq C(1+|x|), \ (T,x)\in [0,\infty)\times\mathbb{R}^n.$$
In particular, it holds true that
$\lim_{T\rightarrow \infty}\frac{V(T,x)}{T}=\rho.$
\end{theorem}
\begin{proof}
Notice that since $w(x)$ is a viscosity solution of the equation \eqref{pde1}, it is also a viscosity solution of the following PDE for any $T>0$,
 \begin{equation}\label{050102}
\left\{
\begin{array}{l}
\frac{\partial}{\partial t}\varphi(t,x)+\inf_{u\in U}\sup_{v\in V}\{H(x,D\varphi(t,x),D^2\varphi(t,x),u,v)-\rho\}=0,\ (t,x)\in[0,T)\times\mathbb{R}^n,\\
\varphi(T,x)=w(x),\ x\in\mathbb{R}^n.
\end{array}
\right.
\end{equation}
Since $w(x)$ at most is of linear growth, it follows from the uniqueness of the viscosity solution of \eqref{050102} (see Theorem 5.3 in \cite{BL2008}) that $w(x)$ is the unique viscosity solution of \eqref{050102}. Moreover, it has the following representation
\begin{equation}\label{050103}
w(x)=\varphi(0,x)=\sup_{\beta\in\mathcal{B}}\inf_{u\in\mathcal{U}}E[\int_{0}^T \big(f(X_s^{x,u,\beta(u)},u_s,\beta(u)_s)-\rho\big) ds+w(X_T^{x,u,\beta(u)})].
\end{equation}
Similar to \eqref{2019041503}, from \eqref{041508} and \eqref{050103},  we have
\begin{equation}\nonumber
|V(T,x)-(\rho T+w(x))|\leq \sup_{\beta\in\mathcal{B}}\sup_{u\in\mathcal{U}}|E[\Phi(X_T^{x,u,\beta(u)})-w(X_T^{x,u,\beta(u)})]|\leq C(1+|x|).
\end{equation}
\end{proof}
\begin{remark}
When the diffusion coefficient is non-degenerate, Alvarez, Bardi \cite{AB2003} obtained a similar convergence result (Proposition 5, \cite{AB2003}) for  a singular perturbation problem by using  PDE techniques. In our case, we do not need the non-degeneracy assumption. We study the long term behaviour of the viscosity solution of equation \eqref{050101} with the help of
Feynman-Kac's formula for the related SDGs.

A result similar to Theorem \ref{character1} is also obtained in \cite{CFP2016} (see Theorems 5.1 and 5.2 therein).  They studied a type of fully nonlinear HJB equations and obtained the long time behaviour of its solutions based on the related stochastic control problem and backward stochastic differential equation theory. This allowed them to avoid the non-degenerate assumption on the diffusion term $\sigma$. Compared with their work, our ergodic HJBI equation is different and totally new. Moreover, we do not need the existence of an invariant measure, which plays an important role in their approach.
\end{remark}

\section{ {\protect \large The existence of the value of our SDGs}}
In this section, without non-degenerate assumption on the diffusion system, we will first prove the upper ergodic value function $\rho^+(x)$ (resp., lower ergodic value function $\rho^-(x)$) to be the scalar quantity $\rho$ which is independent of $x$, where $(\rho,w)$ is a viscosity solution of the ergodic HJBI equation \eqref{pde1} (resp., \eqref{pde2}). Then we derive the uniqueness of the viscosity solution of the ergodic HJBI equations \eqref{pde1} and \eqref{pde2}, and we prove with the help of the uniqueness  results that the value of our SDGs exists under the Isaacs condition.

We first  construct an approximation of the possibly degenerate controlled process $X^{x,u,v}$ by nondegenerate ones. For this, we introduce an auxiliary $n$-dimensional Brownian motion $B^1$ on the underlying probability space $(\Omega,\mathcal{F},P)$, which is independent of the Brownian motion $B$. Moreover, let $W=(B,B^1)$
and we denote by $\mathbb{F}^{B^1}=(\mathcal{F}^{B^1}_t)_{t\geq 0}$, $\mathbb{F}^{W}=(\mathcal{F}^{W}_t)_{t\geq 0}$ the complete filtration generated by the  Brownian motion $(B^1_t)_{t\geq 0}$ and $(W_t)_{t\geq 0}$, respectively.
For $r> 0$, we define the $n\times (d+n)$-matrix $\sigma^r$ as follows:
$$\sigma^r(x,u,v)=(\sigma(x,u,v),rI_{n\times n}),$$
where $I_{n\times n}$ is the identity matrix over $\mathbb{R}^d$. Then we construct a linear SDE associated with  the diffusion coefficient $\sigma^r$ and Brownian motion $W=(B,B^1)$ as follows: for $s\geq 0$,
\begin{equation}\label{dynamics+r}
\left\{
\begin{array}{l}
dX_s^{r,x,u,v}=[-\frac{K}{2}\cdot (X_s^{r,x,u,v}-X_s^{x,u,v})+b(X_s^{x,u,v},u_s,v_s)]ds+\sigma^r(X_s^{x,u,v},u_s,v_s)dW_s,\\
X_0^{r,x,u,v}=x,
\end{array}
\right.
\end{equation}
where $X^{x,u,v}$ is the solution of SDE (\ref{dynamics})
and $K$ is the constant given in assumption (H3).
Then it follows that the (diffusion coefficient of the) controlled process $X^{r,x,u,v}$ is nondegenerate due to
$$\langle\sigma^r(\sigma^r)^*(x,u,v)\xi, \xi\rangle\geq r^2|\xi|^2,\ \text{for\ all}\ \xi\in\mathbb{R}^n.$$
For each $r\in (0,1]$ and all admissible controls $(u,v)\in\mathcal{U}\times\mathcal{V}$, it is clear that  SDE (\ref{dynamics+r}) has a unique solution $X^{r,x,u,v}$.
Moreover, it follows from the Gronwall lemma and Lemma \ref{estimates} that the first estimate in \eqref{051401} still holds for the solution of SDE \eqref{dynamics+r}, i.e., there exist $C>0$ and $c>0$ such that
\begin{equation}\label{112402}
E|X_t^{r,x,u,v}|^2\leq C(1+|x|^2e^{{-c t}}).
\end{equation}
From the uniqueness of the solution of SDE we also conclude that $X^{r,x,u,v}|_{r=0}=X^{x,u,v}$, $(x,u,v)\in\mathbb{R}^n\times\mathcal{U}\times\mathcal{V}$.

Concerning the Brownian motion $B^1$, we regard the game problem with dynamics \eqref{dynamics}, ergodic payoff given by \eqref{equ 2.2} under the underlying filtered probability space $(\Omega,\mathcal{F},(\mathcal{F}_t^B)_{0\leq t\leq T},$ $P)$ as \emph{the original game problem} throughout this section, and we clarify its independence from $B^1$.
{
\begin{remark}\label{re-original admissible controls}
It is worth noting that our admissible controls $(u,v)\in\mathcal{U}\times\mathcal{V}$ are $\mathbb{F}^B$-progressively measurable instead of the $\mathbb{F}^W$-progressively measurable ones used in \cite{S1996}.
The main reason is that we limit ourselves to the original game problem, which seems to have nothing to do with the additional Brownian motion $B^1$ and its generated filtration. The introduction of the Brownian motion $B^1$ is only used as a tool to approximate the original dynamics by a class of non-degenerate processes. If we consider $\mathbb{F}^W$-progressively measurable admissible controls, then the original game problem including dynamics and ergodic payoff, and so the optimal pair of admissible controls (if exists) will be affected by the choice of this additional information generated by filtration $\mathbb{F}^W$.
As a result, our approach will make it easier for players to choose the optimal control in $\mathbb{F}^B$-admissible control sets and it may facilitate  possible numerical computations.
\end{remark}
}
The following lemma shows the convergence of the solution $X^{r,x,u,v}$, when $r$ tends to $ 0$.
\begin{lemma}\label{le1}
Suppose that Assumptions (H1)-(H3) are satisfied. Let $X^{x,u,v}$ and $X^{r,x,u,v}$ be the unique solution of SDE \eqref{dynamics} and \eqref{dynamics+r}, respectively. Then there exists a constant $C$ such that for all $t>0$ and $(x,u,v)\in\mathbb{R}^n\times\mathcal{U}\times\mathcal{V}$,
 \begin{equation}\nonumber
E[|X^{r,x,u,v}_t-X^{x,u,v}_t|^2]\leq Cnr^2.
\end{equation}
\end{lemma}

\begin{proof}
Applying It\^o formula to $|X^{r,x,u,v}_t-X^{x,u,v}_t|^2$ we get directly
$$E|X^{r,x,u,v}_t-X^{x,u,v}_t|^2=\frac{nr^2}{K}(1-e^{-Kt}).$$
\end{proof}
Let us give an upper bound for the density function of the solution $X^{r,x,u,v}$ of SDE \eqref{dynamics+r}.
\begin{proposition}\label{le2}
Suppose that Assumptions (H1)-(H3) hold and let
 $X^{r,x,u,v}$ be the solution of SDE \eqref{dynamics+r}. Then for any Borel set $D\subseteq\mathbb{R}^n$, it holds that for all $(u,v)\in\mathcal{U}\times\mathcal{V}$,
$$P{\{X^{r,x,u,v}_s}\in D\}\leq (\frac{K}{2})^\frac{n}{2} r^{-n}[1-e^{-{K}s}]^{-\frac n 2}Leb(D),$$
where $Leb(D)$ is the Lebesgue measure of the Borel set $D$, and  $K$ is the constant given in Assumption (H3).
\end{proposition}
\begin{remark}\label{re4.1} For any Borel set $D\subseteq\mathbb{R}^n$, we denote by $h$ the density function of the process $X^{r,x,u,v}$, i.e.,
$$P\{X_s^{r,x,u,v}\in D\}=\int_D h(0,x;s,y)dy,\ D\in\mathcal{B}(\mathbb{R}^n).$$
Then, Proposition \ref{le2} says that
$h(0,x;s,y)\leq (\frac{K}{2})^{\frac{n}{2}} r^{-n}[1-e^{-{K}s}]^{-\frac n 2},$ $dy$-a.e. on $\mathbb{R}^n$.
 We remark that this upper bound is weaker than the classical Gaussian's type bound $($see, e.g. \cite{A1967,HLP1997}$)$, which states that
 \begin{equation}\label{2020102601}
h(0,x;s,y)\leq Ms^{-\frac n 2}\exp\{-\frac{N|x-y|^2}{s}\}.
\end{equation}
In our framework, for any given $(u,v)\in\mathcal{U}\times\mathcal{V}$, one may define $$b^{u,v}(t,x)=b(x,u_t,v_t),\ \sigma^{r,u,v}(t,x)=\sigma^r(x,u_t,v_t),$$ which are obviously not deterministic, so the solution $X^{r,x,u,v}$ is not Markovian.
Although the classical PDE approach (see, e.g., \cite{A1967}) corresponding to Markovian process cannot apply directly, Gy\"ongy \cite{G1986} stated that one can construct a corresponding SDE with deterministic coefficients whose solution has the same distribution as that of $X^{r,x,u,v}$. As a result, the estimate \eqref{2020102601} still holds for the density of the solution $X^{r,x,u,v}$. However, we cannot directly use this estimate because
it is obtained for some finite time horizon $[0,T]$ whereas we consider an infinite time horizon problem, as explained in the following Remark \ref{re-Krylov estimate}. It adds here that, even when $u,v$ were deterministic or feedback controls, $X^{t,x,u,v}$ is not Markovian, nevertheless it is the pair $(X^{t,x,u,v},X^{x,u,v})$ that is Markovian in this case.

To this end, we apply a probabilistic method to provide a weaker Gaussian's type upper bound for our density function. Our approach  allows us to consider the upper bound of the fundamental solution of the related parabolic equation with infinite time horizon and stochastic  coefficients $b$ and $\sigma$.
%
\end{remark}
We give the proof of Proposition \ref{le2}.
%
\begin{proof}
{It is easy to verify that the solution $X^{r,x,u,v}$ of SDE \eqref{dynamics+r} can be expressed as (for simplicity, we omit the superscript $\{x,u,v\}$ in the following proof)
\begin{equation}\nonumber
\begin{aligned}
X_s^{r}=&Y_s+\int_0^s re^{\frac{K}{2}(t-s)}dB_t^1,
\end{aligned}
\end{equation}
where
$$Y_s=xe^{-\frac{K}{2}s}+\int_0^se^{\frac{K}{2}(t-s)}[\frac{K}{2}X_t^{x,u,v}+b(X_t^{x,u,v},u_t,v_t)]dt
+\int_0^se^{\frac{K}{2}(t-s)}\sigma(X_t^{x,u,v},u_t,v_t)dB_t.$$
Recalling that each pair of admissible controls $(u,v)$ is $\mathbb{F}^B$-adapted (see \eqref{040801}), we see that $Y$ is $\mathbb{F}^B$-adapted. Then, for every Borel set $D\subseteq\mathbb{R}^n$, we have
\begin{equation}\nonumber
\begin{aligned}
&P\{{X_s^r\in D}\}=E\big[E[I_{\{X_s^r\in D\}}|\mathcal{F}_s^B]\big]=E\big[E[I_{\{y+e^{-\frac{K}{2}s}\cdot Z_s\in D\}}]_{y=Y_s}\big],
\end{aligned}
\end{equation}
where $Z_s=\int_0^sre^{\frac{K}{2}t}dB^1_t$, $s\geq 0$.
Notice that $Z_s\sim N(0,M_s)$, where
$M_s=\frac{r^2}{K}(e^{Ks}-1)I_{n\times n}$.
Therefore,
\begin{equation}\nonumber
\begin{aligned}
E[I_{\{y+ e^{-\frac{K}{2}s}\cdot Z_s\in D\}}]
=&\int_D\frac{1}{(2\pi)^\frac{n}{2}[\det(e^{-Ks}M_s)]^{\frac{1}{2}}}\cdot exp\{-\frac 1 2(z-y)^* e^{Ks}M_s^{-1} (z-y)\}dz\\
\leq &\int_D\frac{1}{\Big[2\pi  \frac{r^2}{K}(1-e^{-{K}s})\Big]^{\frac{n}{2}}}dz
= (\frac2K \pi r^2)^{-\frac n 2}[1-e^{-{K}s}]^{-\frac n 2}Leb(D).
\end{aligned}
\end{equation}}
\end{proof}
We give the estimate for $\rho^+$ and $\rho^-$ in terms of viscosity solutions of ergodic HJBI equations (\ref{pde1}) and (\ref{pde2}), respectively.
{For this, we introduce an additional assumption as follows:
$$\textbf{(H4)}\ \text{The\ diffusion\ coefficient}\ \sigma\ \text{is\  bounded,\ uniformly\ with\ respect\ to}\ (x,u,v)\in\mathbb{R}^n\times U\times V.$$
Throughout this section, we always assume that the Assumptions (H1)-(H4) hold unless indicated.}
\begin{theorem}\label{th1}
 Suppose that $w$ is Lipschitz on $\mathbb{R}^n$, then we have\\
{\rm (i)} If $(\rho,w)$ is a viscosity subsolution of equation (\ref{pde1}) then
\begin{equation}\label{equ 2.4}
\rho\leq \sup_{\beta\in\mathcal{B}}\inf_{u\in\mathcal{U}}\liminf_{T\rightarrow\infty}J(T,x,u,\beta(u)).
\end{equation}
{\rm (ii)} If $(\rho,w)$ is a viscosity supersolution of equation (\ref{pde1}) then
\begin{equation}\label{equ 2.5}
\rho\geq \sup_{\beta\in\mathcal{B}}\inf_{u\in\mathcal{U}}\limsup_{T\rightarrow\infty}J(T,x,u,\beta(u)).
\end{equation}
{\rm (iii)} If $(\rho,w)$ is a viscosity subsolution of equation (\ref{pde2}) then
\begin{equation}\label{equ 2.6}
\rho\leq \inf_{\alpha\in\mathcal{A}}\sup_{v\in\mathcal{V}}\liminf_{T\rightarrow\infty}J(T,x,\alpha(v),v).
\end{equation}
{\rm (iv)} If $(\rho,w)$ is a viscosity supersolution of equation (\ref{pde2}) then
\begin{equation}\label{equ 2.7}
\rho\geq \inf_{\alpha\in\mathcal{A}}\sup_{v\in\mathcal{V}}\limsup_{T\rightarrow\infty}J(T,x,\alpha(v),v).
\end{equation}
\end{theorem}
\begin{proof}
{\bf{Step 1}}. We first prove that this theorem holds for $w\in C^3_b(\mathbb{R}^n)$. We only give the proof of (iii) and (iv), (i) and (ii) can be proved similarly.

\emph{Proof of {\rm (iii)}}. Notice that $w\in C^3_b(\mathbb{R}^n)$, $(\rho,w)$ is the subsolution of the ergodic HJBI equation (\ref{pde2}) in the classical sense, namely
\begin{equation}\label{equ 2.8}
\rho\leq \sup_{v\in V}\Lambda(x,v),\ x\in\mathbb{R}^n,
\end{equation}
where $\Lambda(x,v)=\inf_{u\in U}H(x,Dw(x),D^2w(x),u,v)$.
Since $\Lambda$ is locally Lipschitz in $x$ and continuous with respect to $v$ over the compact space $V$, there exists an $\mathcal{B}(\mathbb{R}^n)$-$\mathcal{B}(V)$ measurable mapping $\psi:\mathbb{R}^n\rightarrow V$ such that,
{for all $x\in\mathbb{R}^n$, $u\in U$,
\begin{equation}\label{equ 2.11}
\rho\leq \Lambda(x,\psi(x))\leq H(x,Dw(x),D^2w(x),u,\psi(x)).
\end{equation}
Let $u\in\mathcal{U}$ be arbitrary but fixed, and let $m$ be a positive integer large enough such that $[0,T]=\bigcup_{i=0}^{m-1}[i\theta,(i+1)\theta]$ with $\theta=\frac{T}{m}$.} We construct an adimissible control $v^{\theta}\in\mathcal{V}$
inductively over time intervals $[i\theta,(i+1)\theta)$, $i=0,1,2,\cdots,$ as follows: For $i=0$, we define
$$v^{\theta}(s)=\psi(x),\ s\in[0,\theta).$$
Then SDE (\ref{dynamics}) has a unique solution $X^{x,u,v^{\theta}}$ over $[0,\theta)$ with $v=v^{\theta}$. With this $X_\theta^{x,u,v^{\theta}}$, we
define
$$v^{\theta}(s)=\psi(X_\theta^{x,u,v^{\theta}}),\ s\in[\theta,2\theta).$$
It is easy to check that $v^{\theta}|_{[\theta,2\theta)}$ is an $\mathcal{F}^B_\theta$-adapted admissible control. We can then solve the SDE (\ref{dynamics}) on $[\theta,2\theta)$ and obtain $X_{2\theta}^{x,u,v^{\theta}}.$ Repeating the above process, we construct an admissible control
$$v^{\theta}(s)=\psi(X_{i\theta}^{x,u,v^{\theta}}),\ \text{if}\ s\in[i\theta,(i+1)\theta),$$
for Player 2.
Applying It\^o's formula to $w(X_s^{x,u,v^{\theta}})$, we get
\begin{equation}\label{equ 2.12}
\begin{aligned}
&E[w(X_{(i+1)\theta}^{x,u,v^{\theta}})-w(X_{i\theta}^{x,u,v^{\theta}})]\\
=&E\int_{i\theta}^{(i+1)\theta}\langle b(X_s^{x,u,v^{\theta}},u_s,
\psi(X_{i\theta}^{x,u,v^{\theta}})), Dw(X_s^{x,u,v^{\theta}})\rangle\\
&\qquad\qquad+\frac{1}{2}tr\big((\sigma\sigma^*)
(X_s^{x,u,v^{\theta}},u_s,\psi(X_{i\theta}^{x,u,v^{\theta}}))\cdot D^2w(X_s^{x,u,v^{\theta}})\big)ds\\
\end{aligned}
\end{equation}

\begin{equation}\nonumber
\begin{aligned}
=&E\int_{i\theta}^{(i+1)\theta}\Big(\langle b(X_s^{x,u,v^{\theta}},u_s,
\psi(X_{i\theta}^{x,u,v^{\theta}})), Dw(X_s^{x,u,v^{\theta}})\rangle\\
&\qquad\quad-\langle b(X_{i\theta}^{x,u,v^{\theta}},u_s,
\psi(X_{i\theta}^{x,u,v^{\theta}})), Dw(X_{i\theta}^{x,u,v^{\theta}})\rangle\\
&\qquad\quad+\frac{1}{2}tr\big((\sigma\sigma^*)(X_s^{x,u,v^{\theta}},
u_s,\psi(X_{i\theta}^{x,u,v^{\theta}}))\cdot D^2w(X_s^{x,u,v^{\theta}})\big)\\
&\qquad\quad-\frac{1}{2}tr\big((\sigma\sigma^*)(X_{i\theta}^{x,u,v^{\theta}},
u_s,\psi(X_{i\theta}^{x,u,v^{\theta}}))\cdot D^2w(X_{i\theta}^{x,u,v^{\theta}})\big)\\
&\qquad\quad+f(X_s^{x,u,v^{\theta}},u_s,\psi(X_{i\theta}^{x,u,v^{\theta}}))
-f(X_{i\theta}^{x,u,v^{\theta}},u_s,\psi(X_{i\theta}^{x,u,v^{\theta}}))\Big)ds\\
&+E\int_{i\theta}^{(i+1)\theta}H(X_{i\theta}^{x,u,v^{\theta}},
Dw(X_{i\theta}^{x,u,v^{\theta}}),D^2w(X_{i\theta}^{x,u,v^{\theta}}),u_s,
\psi(X_{i\theta}^{x,u,v^{\theta}}))ds\\
&-E\int_{i\theta}^{(i+1)\theta}f(X_s^{x,u,v^{\theta}},u_s,\psi(X_{i\theta}^{x,u,v^{\theta}}))ds.
\end{aligned}
\end{equation}
Combining Assumption (H4), (\ref{equ 2.11}) and \eqref{051401},  it follows from \eqref{equ 2.12} that
\begin{equation}\label{equ 2.13}
\begin{aligned}
&E[w(X_{(i+1)\theta}^{x,u,v^{\theta}})-w(X_{i\theta}^{x,u,v^{\theta}})]\\
\geq& -C\theta \big(E[\sup_{i\theta \leq s\leq (i+1)\theta}|X_s^{x,u,v^{\theta}}-X_{i\theta}^{x,u,v^{\theta}}|^2]\big)^\frac 1 2+
\rho\cdot\theta
-E\int_{i\theta}^{(i+1)\theta}f(X_s^{x,u,v^{\theta}},u_s,
\psi(X_{i\theta}^{x,u,v^{\theta}}))ds\\
\geq& -C\theta (\theta+\theta^{\frac{1}{2}})+
\rho\cdot\theta-E\int_{i\theta}^{(i+1)\theta}f(X_s^{x,u,v^{\theta}},u_s,
\psi(X_{i\theta}^{x,u,v^{\theta}}))ds.
\end{aligned}
\end{equation}
Taking $\sum\limits_{i=0}^{m-1}$ on both sides of (\ref{equ 2.13}), we obtain
\begin{equation}\label{equ 2.14}
E[w(X_{T}^{x,u,v^{\theta}})-w(x)]
\geq -CT (\theta+\theta^{\frac{1}{2}})+
\rho\cdot T-E\int_{0}^{T}f(X_s^{x,u,v^{\theta}},u_s,v_s^{\theta})ds.
\end{equation}
Let $m=[T^2]$. Thus, dividing by $T$ on both sides and letting $T\rightarrow \infty$, we have
\begin{equation}\label{equ 2.16}
\rho\leq \liminf_{T\rightarrow\infty}J(T,x,u,v^{\theta}).
\end{equation}
Given any $\alpha\in\mathcal{A}$,  we construct an admissible strategy $\beta^{\theta}\in\mathcal{B}$ and a pair of admissible controls $(\tilde{u}^{\theta},\tilde{v}^{\theta})\in\mathcal{U}\times\mathcal{V}$ inductively on the intervals $[i\theta,(i+1)\theta)$, $i=1,2,\cdots.$ For $ s\in[0,\theta)$, we define
$$\tilde{v}^{\theta}_s=\psi(x),\ \tilde{u}^{\theta}_s=\alpha(\tilde{v}^{\theta})_s,\ \beta^{\theta}(u)_s=\tilde{v}^{\theta}_s,\ \text{for}\ u\in\mathcal{U}.$$
This allows us to obtain $X^{x,\tilde{u}^{\theta},\tilde{v}^{\theta}}|_{[0,\theta)}$, the solution of SDE \eqref{dynamics} associated with $\tilde{u}^{\theta},\tilde{v}^{\theta}$ over the time interval $[0,\theta)$. With $\beta^{\theta},\tilde{u}^{\theta},\tilde{v}^{\theta}$ and $X^{x,\tilde{u}^{\theta},\tilde{v}^{\theta}}$ on the interval $[0,i\theta)$, we define, for $u\in\mathcal{U}$,
\begin{equation}\label{112001} \tilde{v}^{\theta}_s
=\psi(X_{i\theta}^{x,\tilde{u}^{\theta}|_{[0,i\theta)},
\tilde{v}^{\theta}|_{[0,i\theta)}}),\  \tilde{u}^{\theta}_s=\alpha(\tilde{v}^{\theta})_s,\
\beta^{\theta}(u)_s=\tilde{v}^{\theta}_s,\ s\in[i\theta,(i+1)\theta).
\end{equation}
It is easy to check that  $\beta^{\theta}\in\mathcal{B}$, $(\tilde{u}^{\theta},\tilde{v}^{\theta})\in\mathcal{U}\times\mathcal{V}$ and that the following relation holds
\begin{equation}\label{052101}
\alpha(\tilde{v}^{\theta})=\tilde{u}^{\theta},\ \beta^{\theta}(\tilde{u}^{\theta})=\tilde{v}^{\theta},\ dsdP\text{-a.e., on}\ [0,\infty).
\end{equation}
Then, for any $\alpha\in\mathcal{A}$, from (\ref{equ 2.16}) with choosing $u=\tilde{u}^{\theta}$ (thus $v^{\theta}$ in \eqref{equ 2.16}  equals $\tilde{v}^{\theta}$, see \eqref{112001}), we conclude that
\begin{equation}\nonumber
 \rho\leq \liminf_{T\rightarrow\infty}
 J(T,x,\alpha(\tilde{v}^{\theta}),\tilde{v}^{\theta})
 \leq \sup_{v\in\mathcal{V}}\liminf_{T\rightarrow\infty}J(T,x,\alpha(v),v).
\end{equation}
From the arbitrariness of $\alpha$, we get
\begin{equation}\nonumber
 \inf_{\alpha\in\mathcal{A}}\sup_{v\in\mathcal{V}}\liminf_{T\rightarrow\infty}J(T,x,\alpha(v),v)
 \geq \rho.
\end{equation}

\noindent\emph{Proof of {\rm (iv)}}. Similarly to (\ref{equ 2.8}), we have
\begin{equation}\nonumber
\rho\geq \sup_{v\in V}\inf_{u\in U}H(x,Dw(x),D^2w(x),u,v)\geq \inf_{u\in U}H(x,Dw(x),D^2w(x),u,v),\ (x,v)\in\mathbb{R}^n\times V.
\end{equation}
Then there exists a measurable  $\psi:\mathbb{R}^n\times V\rightarrow U$ such that,
for all $v\in V$, $x\in\mathbb{R}^n$,
\begin{equation}\nonumber
\rho\geq H(x,Dw(x),D^2w(x),\psi(x,v),v).
\end{equation}
This is similar to \eqref{equ 2.11}, but here the mapping $\psi$ also relies on the control variable $v$.
Let $\theta>0$ be a small constant. Similarly to the construction of $\beta^{\theta}$ in the proof of (iii),
we  construct $\alpha^{\theta}\in\mathcal{A}$ as follows. Define
$$\alpha^{\theta}(v)_s=\psi(x,v_s),\ s\in[0,\theta],\ v\in \mathcal{V}.$$
Then we have the unique solution $X^{x,\alpha^{\theta}(v),v}$ of SDE (\ref{dynamics}) on $[0,\theta)$.
Assuming that both strategy $\alpha^{\theta}$ and solution $X^{x,\alpha^{\theta}(v),v}$ have been obtained on $[0,i\theta]$, we extend their definition to the interval  $[0,(i+1)\theta]$. For this, we define
$$\alpha^{\theta}(v)_s=\psi(X^{x,\alpha^{\theta}(v),v}_{i\theta},v_s),\ s\in[i\theta,(i+1)\theta],\ v\in\mathcal{V}.$$
It is easy to check that $\alpha^{\theta}\in\mathcal{A}$.
Similarly to (\ref{equ 2.16}) summing up over $1\leq i\leq m=[T^2]$ and taking the limit as $T\rightarrow\infty$, we get that for any $v\in\mathcal{V}$,
\begin{equation}\nonumber
\rho\geq \limsup_{T\rightarrow\infty}J(T,x,\alpha^{\theta}(v),v).
\end{equation}
From the arbitrariness of $v\in\mathcal{V}$, we have
\begin{equation}\nonumber
\rho\geq \inf_{\alpha\in\mathcal{A}}\sup_{v\in\mathcal{V}}\limsup_{T\rightarrow\infty}J(T,x,\alpha(v),v).
\end{equation}

\noindent {\bf{Step 2}}. We now consider the general case, i.e., $w$ is only Lipschitz on $\mathbb{R}^n$. We give the proof of (iii) only, while the arguments for (i), (ii), and (iv) are analogous.

Let $(\rho,w)$ be a viscosity subsolution of (\ref{pde2}).
For $\varepsilon\in(0,1]$, let $w^\varepsilon$ be the sup-convolution of $w$, that is,
\begin{equation}\label{equ 2.20}
w^\varepsilon(x)=\sup_{y\in\mathbb{R}^n}\{w(y)-\frac{|x-y|^2}{2\varepsilon}\}.
\end{equation}
Since $w$ is Lipschitz, it is straightforward to check that $w^\varepsilon$ converges to $w$ uniformly in $\mathbb{R}^n$ as $\varepsilon\rightarrow 0$ and
$w^\varepsilon$ is of linear growth for any fixed $\varepsilon>0$, Lipschitz continuous, semiconvex with semiconvexity constant $\frac{1}{2\varepsilon}$ (i.e., $w^\varepsilon(x)+\frac{|x|^2}{2\varepsilon}$ is convex) (see \cite{S1996} for more details, or also \cite{LL1986}). It follows from Lemma \ref{Appendix} in the Appendix that $(\rho,w^\varepsilon)$ is a viscosity subsolution of the following PDE
\begin{equation}\label{equ 022801}
\rho=\sup_{v\in V}\inf_{u\in U}H(x,Dw^\varepsilon(x),D^2w^\varepsilon(x),u,v)
+{k(\varepsilon)},
\end{equation}
where $k(\varepsilon)\rightarrow 0$, as $\varepsilon\rightarrow 0$.
Since $w^\varepsilon$ is semiconvex, from Theorem A.2 in \cite{CIL1992}, we see that $w^\varepsilon$ is twice differentiable almost
everywhere on $\mathbb{R}^n$. This indicates that $w^\varepsilon$ satisfies a.e. on $\mathbb{R}^n$
\begin{equation}\label{equ 2.25}
\rho\leq \sup_{v\in V}\inf_{u\in U}H(x,Dw^\varepsilon(x),D^2w^\varepsilon(x),u,v)+
{k(\varepsilon)}.
\end{equation}
For any given $\delta>0$, let $w^\varepsilon_\delta$ be the smooth approximation of $w^\varepsilon$, i.e.,
\begin{equation}\label{equ 2.26}
w^\varepsilon_\delta(x)=\frac{1}{\delta^n}\int_{\mathbb{R}^n} w^\varepsilon(y)\cdot\varphi(\frac{x-y}{\delta})dy,
\end{equation}
where $\varphi$ is a standard mollifier in $\mathbb{R}^n$ with compact support and satisfying $\int_{\mathbb{R}^n}\varphi(x)dx=1$.
Then  $w^\varepsilon_\delta\in C_b^3(\mathbb{R}^n)$, $w^\varepsilon_\delta$ converges uniformly to $w^\varepsilon$  on $\mathbb{R}^n$, and $Dw^\varepsilon_\delta$, $D^2w^\varepsilon_\delta$ converge to $Dw^\varepsilon$, $D^2w^\varepsilon$ for a.e. $x\in\mathbb{R}^n$, respectively, as $\delta\rightarrow 0$. Moreover, $w_\delta^\varepsilon$ has the same Lipschitz and semiconvexity constants as $w^\varepsilon$.
Then,
 we find that $w_\delta^\varepsilon$ satisfies, on $\mathbb{R}^n$
\begin{equation}\label{equ 2.27}
\rho\leq \sup_{v\in V}\inf_{u\in U}H(x,Dw^\varepsilon_\delta(x),D^2w^\varepsilon_\delta(x),u,v)
+{k(\varepsilon)}+{g_\delta(x)},
\end{equation}
where
\begin{equation}\label{112401}
g_\delta(x)=\Big\{\rho- \sup_{v\in V}\inf_{u\in U}H(x,Dw^\varepsilon_\delta(x),D^2w^\varepsilon_\delta(x),u,v)
-{k(\varepsilon)}\Big\}^+.
\end{equation}
Then it is easy to check from  $w_\delta^\varepsilon\in C_b^3(\mathbb{R}^n)$, \eqref{equ 2.25}, and \eqref{112401} that $(g_\delta)_{\delta>0}$  are continuous in $x$, of linear growth in $x$, uniformly in $\delta$ (due to the uniform semiconvexity and uniform Lipschitz continuity of $w^\varepsilon_\delta$),
 and $g_\delta(x)\rightarrow 0$, as $\delta\rightarrow 0$, for a.e. $x\in\mathbb{R}^n$.
Without loss of generality, we assume that $g_\delta$ is Lipschitz continuous. Otherwise we can construct a sequence of Lipschitz continuous functions to approximate $g_\delta$ (see Lemma 3.1 in \cite{LPL2016}).
For $(x,p,A,u,v)\in \mathbb{R}^n\times\mathbb{R}^n\times \mathcal{S}^{n\times n}\times U\times V$, we denote
$$H^{\delta,\varepsilon}(x,p,A,u,v)={\frac{1}{2}tr\big((\sigma\sigma^{*})(x,u,v)\cdot A\big)}+\langle b(x,u,v), p\rangle+f^{\delta,\varepsilon}(x,u,v),$$
where $f^{\delta,\varepsilon}(x,u,v)=f(x,u,v)+{k(\varepsilon)}
+{g_\delta(x)}.$
Then it follows from (\ref{equ 2.27}) that
\begin{equation}\label{equ 030801}
\rho\leq \sup_{v\in V}\inf_{u\in U}H^{\delta,\varepsilon}(x,Dw^\varepsilon_\delta(x),D^2w^\varepsilon_\delta(x),u,v).
\end{equation}
Thus, applying the results of Step 1 to $w_\delta^\varepsilon$,
we obtain from (\ref{equ 030801}), that
\begin{equation}\label{2022041301}
\begin{aligned}
\rho&\leq\inf_{\alpha\in\mathcal{A}}\sup_{v\in\mathcal{V}}\liminf_{T\rightarrow\infty}\frac{1}{T}
E[\int_0^Tf^{\delta,\varepsilon}({{X_s^{x,\alpha(v),v}},\alpha(v)_s,v_s})ds]\\
&=\inf_{\alpha\in\mathcal{A}}\sup_{v\in\mathcal{V}}
\liminf_{T\rightarrow\infty}\frac{1}{T}
E[\int_0^T\big(f({X_s^{x,\alpha(v),v}},\alpha(v)_s,v_s)
+{g_\delta(X_s^{x,\alpha(v),v})}\big)ds]
+k(\varepsilon)\\
&\leq \inf_{\alpha\in\mathcal{A}}\sup_{v\in\mathcal{V}}
\liminf_{T\rightarrow\infty}\frac{1}{T}
E[\int_0^T\big(f({X_s^{x,\alpha(v),v}},\alpha(v)_s,v_s)
+{g_\delta(X_s^{r,x,\alpha(v),v})}\big)ds]
+C\sqrt{n}r+k(\varepsilon),
\end{aligned}
\end{equation}
where the last inequality is obtained from Lemma \ref{le1}.
We now give the estimate to the term
\begin{equation}\nonumber
\frac{1}{T}E[\int_0^Tg_\delta(X_s^{r,x,\alpha(v),v})ds].
\end{equation}
Let  $R>0$ be arbitrarily fixed and $E_R$ be a Borel subset of $\mathbb{R}^n$ with $Leb(E_R)\leq \frac 1 R$, such that $g_\delta\rightarrow 0$ uniformly on $B_R(0)\backslash E_R$. Here $B_R(0)$ denotes the open ball of radius $R$ centred at $0$. In fact, from \eqref{112401} it is easy to check that there exists some modulus $h(\delta)$ such that {$|g_\delta(x)|\leq h(\delta)(1+|x|)$} on $B_R(0)\backslash E_R$ and $h(\delta)\rightarrow 0$, as $\delta\rightarrow 0$.
Noting that $g_\delta$ is of linear growth, from the Markov inequality and \eqref{112402} it follows that there exists a constant $C>0$ (independent of the controls $u$ and $v$) such that
\begin{equation}\label{051301}
\begin{aligned}
&\frac 1 T E[\int_0^TI_{\{|X_s^{r,x,\alpha(v),v}|>{R}\}}\cdot g_\delta({X_s^{r,x,\alpha(v),v}})ds]
\leq {\frac C T E[\int_0^TI_{\{{|X_s^{r,x,\alpha(v),v}|>{R}}\}}\cdot (1+|{X_s^{r,x,\alpha(v),v}}|)ds]}\\
\leq &{\frac C T\Big(\int_0^TP(|X_s^{r,x,\alpha(v),v}|>R)ds\Big)^\frac 1 2\Big(\int_0^T (1+E|{X_s^{r,x,\alpha(v),v}|^2})ds\Big)^\frac 1 2} \leq \frac{C}{R}.
\end{aligned}
\end{equation}
From Propositions \ref{le2} and \eqref{112402} we get that for any $T>1$,
\begin{equation}\label{equ 032502}
\begin{aligned}
&\frac{1}{T}E[\int_0^TI_{\{X_s^{r,x,\alpha(v),v}\in E_{R}\}}\cdot g_\delta(X_s^{r,x,\alpha(v),v})ds]\\
\leq& \frac{1}{T}E[\int_0^1I_{\{X_s^{r,x,\alpha(v),v}\in E_{R}\}}\cdot g_\delta(X_s^{r,x,\alpha(v),v})ds]+\frac{1}{T}E[\int_1^TI_{\{X_s^{r,x,\alpha(v),v}\in E_{R}\}}\cdot g_\delta(X_s^{r,x,\alpha(v),v})ds]\\
\leq& \frac{C}{T}E[\int_0^1 (1+|X_s^{r,x,\alpha(v),v}|)ds]+{\frac{C}{T}\int_1^TE[I_{\{X_s^{r,x,\alpha(v),v}\in E_{R}\}}\cdot(1+|X_s^{r,x,\alpha(v),v}|)]ds}\\
\leq&\frac{C}{T}+{\frac{C}{T}\int_1^T\Big(P\{X_s^{r,x,\alpha(v),v}\in E_{R}\}\Big)^\frac 1 2 ds}
\leq \frac CT+\frac{C}{r^{\frac n 2}R^\frac12}.
\end{aligned}
\end{equation}
 Therefore, from \eqref{051301} and \eqref{equ 032502}, we finally get
\begin{equation}\label{051302}
\begin{aligned}
&\frac 1 T E[\int_0^Tg_\delta(X_s^{r,x,\alpha(v),v})ds]
\leq \frac{C}{R}+\frac 1 T E[\int_0^TI_{\{|X_s^{r,x,\alpha(v),v}|\leq {R}\}}\cdot g_\delta(X_s^{r,x,\alpha(v),v})ds]\\
 \leq& \frac{C}{R}+\frac 1 T E[\int_0^TI_{\{X_s^{r,x,\alpha(v),v}\in B_R(0)\backslash E_R\}}\cdot h(\delta)\cdot{(1+|X^{r,x,\alpha(v),v}_s|)}ds]+\frac{C}{T}+\frac{C}{r^{\frac n 2}R^\frac12}\\
 \leq&  \frac{C}{R}+ Ch(\delta)+\frac{C}{T}+\frac{C}{r^{\frac n 2}R^\frac12}.
\end{aligned}
\end{equation}
Substituting \eqref{051302} into \eqref{2022041301}, we finally get
\begin{equation}\nonumber
\begin{aligned}
\rho\leq\inf_{\alpha\in\mathcal{A}}\sup_{v\in\mathcal{V}}
\liminf_{T\rightarrow\infty}\frac{1}{T}
E[\int_0^Tf(X_s^{x,\alpha(v),v},\alpha(v)_s,v_s)ds]
+{\frac{C}{R}+ Ch(\delta)+\frac{C}{r^{\frac n 2}R^\frac12}
+C\sqrt{n}r+k(\varepsilon)}.
\end{aligned}
\end{equation}
By first letting $R\rightarrow\infty$ and then sending $\delta,r,\varepsilon\rightarrow 0$ sequentially, we obtain
\begin{equation}\nonumber
\begin{aligned}
\rho\leq\inf_{\alpha\in\mathcal{A}}\sup_{v\in\mathcal{V}}\liminf_{T\rightarrow\infty}\frac{1}{T}E[\int_0^Tf(X_s^{x,\alpha(v),v},\alpha(v)_s,v_s)ds]=\inf_{\alpha\in\mathcal{A}}\sup_{v\in\mathcal{V}}\liminf_{T\rightarrow\infty}
J(T,x,\alpha(v),v).
\end{aligned}
\end{equation}
\end{proof}

\begin{remark}\label{re-Krylov estimate}
We comment that the basic idea of the proof of Theorem \ref{th1} for the ergodic case is adapted from that of Theorem 2.1 in \cite{S1996} for the discounted case (the so-called Abel mean). However, there are some differences between these two situations. First of all, our construction of a type of non-degenerate diffusion processes is a family of linear SDEs, which is different from the approach used in \cite{S1996}. Secondly, our admissible controls are allowed to be progressively measurable with respect to the original filtration generated by the Brownian motion $B$, as explained in Remark \ref{re-original admissible controls}. Thirdly, their assumptions about the coefficients $b,\sigma,f$ in \cite{S1996} are stronger than those required in this paper. Finally, the approach used in \cite{S1996} cannot be applied directly to our ergodic case. Indeed,
note that estimate \eqref{equ 032502} is derived by using Proposition 4.1, while a similar estimate is obtained in \cite{S1996} (see the last line on Page 11 there) for a finite horizon situation  via a well-known estimate by  Krylov of the distribution of a stochastic integral (see, Theorem 4 on Page 66 in \cite{K1980}).
 We emphasise here that  Krylov's estimate does not work for our ergodic SDGs framework.
In fact, recall that the classical Krylov estimate states that for all Borel functions $g(x)$, it holds that, for $p\geq n$ and all Borel sets $E_R$,
\begin{equation}\nonumber
E[\int_0^\infty e^{-\lambda t}\cdot I_{\{X_s^{r,x,\alpha(v),v}\in E_R\}}\cdot|g(X_s^{r,x,\alpha(v),v})|ds]\leq N\cdot (\int_{E_R}|g(x)|^pdx)^{\frac1p},
\end{equation}
where the constant $N$ depends only on the constants $n, p, \lambda, R $ and $r$. Then this estimate is applied in  \cite{S1996} to a finite time horizon case. However, this estimate does not work anymore for the term
\begin{equation}\nonumber
\lim_{T\rightarrow\infty}\frac{1}{T}E[\int_0^T  I_{\{X_s^{r,x,\alpha(v),v}\in E_R\}}\cdot|g(X_s^{r,x,\alpha(v),v})|ds]
\end{equation}
which is equivalent to (see Remark \ref{re-AT-theorem})
\begin{equation}\nonumber
\lim_{\lambda\rightarrow 0}\lambda E[\int_0^\infty e^{-\lambda t} I_{\{X_s^{r,x,\alpha(v),v}\in E_R\}}\cdot|g(X_s^{r,x,\alpha(v),v})|ds],
\end{equation}
since the constant $N$ depending on $\lambda$ and $\lambda N$ could tend to infinity as $\lambda$ goes to $0.$
\end{remark}
As an immediate consequence of Theorem \ref{th1} we obtain the following comparison
principle and the uniqueness result for the ergodic HJBI equation (\ref{pde2}). A similar result holds  for the ergodic HJBI equation (\ref{pde1}).

\begin{corollary}\label{co1}
{\rm (i)} Let $(\rho_1,w_1)$ and $(\rho_2,w_2)$ be  a viscosity subsolution and supersolution of
(\ref{pde2}), respectively. If $w_1$ and $w_2$ are Lipschitz, it holds true that $\rho_1\leq \rho_2$.\\
{\rm (ii)} Let $(\rho_1,w_1)$ and $(\rho_2,w_2)$ be two viscosity solutions of
(\ref{pde2}), then $\rho_1=\rho_2$. Moreover, assume that {$w_1$ and $w_2$ are Lipschitz, and for some positive constant $R$ (which will be refined in Remark \ref{re-uniqueness}),}
\begin{equation}\label{2020122401}
w_1= w_2,\ \text{on}\ \bar{B}_{R}(0).
\end{equation}
Then $w_1= w_2$ on $\mathbb{R}^n$.
\end{corollary}
\begin{proof}
We only give the proof of (ii) since statement (i) can be obtained directly from Theorem \ref{th1}. The uniqueness of $\rho$ is obtained from (i), so we only need to prove $w_1=w_2$. For this, we consider only the case $w_1,w_2\in C_b^3(\mathbb{R}^n)$, the general case follows similarly as in Step 2 of the proof of Theorem \ref{th1}.

{We now only need to prove $w_1(x)\leq w_2(x)$, on $\mathbb{R}^n$, when $(\rho,w_1)$ (resp. $(\rho,w_2)$) is a subsolution (resp. supersolution) of \eqref{pde2}, since it is easy to check that $w_1(x)\geq w_2(x)$ by using a symmetric argument.
For simplicity, we denote $w(x)=w_1(x)-w_2(x)$, $x\in\mathbb{R}^n$. To complete the proof, we are going to show
\begin{equation}\label{2021013101}
w(x)\leq \epsilon |x|^2,\ x\in\mathbb{R}^n,
\end{equation}
for any $\epsilon>0$. Then we obtain the desired result by letting $\epsilon\rightarrow 0.$}

{Notice that $w(x)\equiv 0$, when $|x|\leq R$, we only need to show that \eqref{2021013101} holds when $|x|\geq R.$ Since $w(x)-\epsilon |x|^2$ is bounded from above when $|x|\geq R$, there exists $x_0\in\mathbb{R}^n$ with $|x_0|\geq R$ such that
\begin{equation}\label{2021013102}
w(x_0)-\epsilon |x_0|^2=\max_{ x\in (B_R(0))^c}(w(x)-\epsilon |x|^2).
\end{equation}
We claim that the point $x_0$ satisfies $|x_0|=R$. Otherwise, if the point $x_0$ satisfies $|x_0|>R$  it would follow from \eqref{2021013102}  that
$$D(w(x)-\epsilon |x|^2)|_{x=x_0}=0,$$
$$ D^2(w(x)-\epsilon |x|^2)|_{x=x_0}\leq 0.$$
Therefore, we have
\begin{equation}\label{2021013103}
\begin{aligned}
&\sup_{u\in U,\ v\in V}\{\frac{1}{2}tr(\sigma\sigma^*(x,u,v)D^2(w(x)-\epsilon |x|^2))+\langle b(x,u,v),D(w(x)-\epsilon |x|^2)\rangle\}_{x=x_0}\\
=&\sup_{u\in U,\ v\in V}\{\frac{1}{2}tr(\sigma\sigma^*(x,u,v)D^2(w(x)-\epsilon |x|^2))\}_{x=x_0}\leq 0,\ \text{with}\ |x_0|>R.
\end{aligned}
\end{equation}
On the other hand, since $(\rho,w_1)$ (resp. $(\rho,w_2)$) is a subsolution (resp. supersolution) of \eqref{pde2}, we get
$$
\rho\leq \sup_{ v\in V}\inf_{u\in U}\{\frac{1}{2}tr(\sigma\sigma^*(x,u,v)D^2w_1(x) )+\langle b(x,u,v),Dw_1(x)\rangle+f(x,u,v)\},$$
and
$$\rho\geq \sup_{ v\in V}\inf_{u\in U}\{\frac{1}{2}tr(\sigma\sigma^*(x,u,v)D^2w_2(x) )+\langle b(x,u,v),Dw_2(x)\rangle+f(x,u,v)\},
$$
which means that, for all $x\in\mathbb{R}^n$,
\begin{equation}\label{2021013104}
 \sup_{u\in U, v\in V}\{\frac{1}{2}tr(\sigma\sigma^*(x,u,v)D^2w(x) )+\langle b(x,u,v),Dw(x)\rangle\}\geq 0.
\end{equation}
Noting that $D|x|^2=2x,$ $D^2|x|^2=2 I_{n\times n}$, there exists a constant $R>0$, such that for all $x\in\mathbb{R}^n$ with $|x|> R$,
\begin{equation}\label{2021013105}
\begin{aligned}
& \sup_{u\in U, v\in V}\{\frac{1}{2}tr(\sigma\sigma^*(x,u,v)D^2|x|^2 )+\langle b(x,u,v),D|x|^2\rangle\}\\
 =& \sup_{u\in U, v\in V}\{\|\sigma(x,u,v)\|^2+2\langle x,(b(x,u,v)-b(0,u,v))\rangle+2\langle x,b(0,u,v)\rangle\}\\
\leq & {\tilde{\sigma}^2-K| x|^2+2|x|\tilde{b}}
<0,\\
 \end{aligned}
\end{equation}
where $\tilde{b}:=\sup_{(u,v)\in{U}\times V}|b(0,u,v)|$, $\tilde{\sigma}:=\sup_{(x,u,v)\in\mathbb{R}^n\times{U}\times V}\|\sigma(x,u,v)\|$. Combining \eqref{2021013104} and \eqref{2021013105}, we get
\begin{equation}\nonumber
 \sup_{u\in U, v\in V}\{\frac{1}{2}tr(\sigma\sigma^*(x,u,v)D^2(w(x)-\epsilon |x|^2) )+\langle b(x,u,v),D(w(x)-\epsilon |x|^2)\rangle\}> 0,\ \text{for all}\ |x|> R,
\end{equation}
which contradicts \eqref{2021013103}. Therefore, $x_0$ satisfies $|x_0|=R$.}

Finally, from \eqref{2021013102} we conclude that, for all $|x|\geq R$,
$$w(x)-\epsilon |x|^2\leq w(x_0)-\epsilon R^2=-\epsilon R^2<0,$$
which implies that $w(x)<\epsilon |x|^2$. Letting $\epsilon\downarrow 0$, we get $w_1(x)\leq w_2(x),\ x\in\mathbb{R}^n.$
\end{proof}
\begin{remark}
We remark that the idea of the proof of Corollay \ref{co1} is inspired by the uniqueness proof given in \cite{BBP} (see the proof of Theorem 3.5 therein), where they divided the proof into the following three steps:\\
\emph{(1)} Show that $w(x)=w_1(x)-w_2(x)$ is a viscosity subsolution of the associated partial differential system;\\
\emph{(2)} Build a suitable  smooth supersolution $\chi(x)$ of this system;\\
\emph{(3)} Show that  $w(x)\leq \varepsilon \chi(x)$ in $\mathbb{R}^n$.\\
 Different from their work, we consider a type of ergodic HJBI equations and apply a family of functions with adequate regularities to   approximate our viscosity solution. As a result, we only need to consider the classical subsolution $w_1(x)$ $($resp. supersolution $w_2(x)$$)$ such that $w\in C_b^3(\mathbb{R}^n)$ and conclude directly the result $w(x)\leq \varepsilon \chi(x)$ in $\mathbb{R}^n$ by choosing $\chi(x)=|x|^2$. In this sense,
 our uniqueness proof looks more simplified than those used in \cite{BBP}.
\end{remark}
\begin{remark}\label{re-uniqueness}
It is easy to verify from \eqref{2021013105} that constant $R$ assumed in Corollay \ref{co1} can be taken
\begin{equation}\label{060601}
R:=\frac{\tilde{b}+\sqrt{\tilde{b}^2+K\tilde{\sigma}^2}}{K},
\end{equation}
with $\tilde{b}=\sup_{(u,v)\in{U}\times V}|b(0,u,v)|$, $\tilde{\sigma}=\sup_{(x,u,v)\in\mathbb{R}^n\times{U}\times V}\|\sigma(x,u,v)\|$.
Obviously, the constant $R$ depends only on the constants $\tilde{b}$, $\tilde{\sigma}$, $K$, and
the constant $R$ will be smaller when the constant $K$ increases.


It is worth pointing out that an additional condition {\eqref{2020122401}}
is necessary to obtain the uniqueness result for the ergodic HJBI equation \eqref{pde2}. In fact, for all constants $C$,  it is obvious that $(\rho,w+C)$ are the classical solutions of the ergodic HJBI \eqref{pde2} if $(\rho,w)$ is a classical solution.

On the other hand, when the diffusion $\sigma(x,u,v)\equiv 0$, our SDGs will reduce  to a class of deterministic differential games, and the ergodic HJBI equation \eqref{pde2} will be a type of HJI equation. In this situation, condition \eqref{2020122401} will be written as
\begin{equation}\label{060603}
w_1=w_2\ \text{on}\ \bar{B}_R(0),\ \text{with}\ R=\frac{2\tilde{b}}{K},
\end{equation}
 since $\tilde{\sigma}=0$ in \eqref{060601}.
Note that Ghosh and  Rao \cite{GR2005}  imposed another necessary condition for the deterministic case, which states that
\begin{equation}\label{060602}
w_1=w_2\ \text{on\ the\ set}\ Z=\{z:z=\lim_{t_n\rightarrow\infty}X_{t_n}^{x,u,v},\ (u,v)\in U\times V\},
\end{equation}
in order to ensure the uniqueness of the viscosity solution of the related HJI equation. We emphasise that our sufficient condition \eqref{060603} cannot directly imply their condition \eqref{060602} and vice versa in this special situation.

We also remark that the condition \eqref{2020122401} for the general situation $\sigma\neq 0$ may not be the weakest condition to guarantee the uniqueness result for such type of ergodic HJBI equations.  We leave it as a future research project to improve this condition.
\end{remark}
\begin{theorem}\label{th2}
Suppose that the classical Isaacs condition
$$\inf_{u\in U}\sup_{v\in V}H(x,p,A,u,v)=\sup_{v\in V}\inf_{u\in U}H(x,p,A,u,v),\ (x,p,A)\in\mathbb{R}^n\times\mathbb{R}^n\times\mathcal{S}(n),$$
holds. Then our SDG with ergodic payoff has a value.
\end{theorem}
\begin{proof}
From Theorem \ref{th3.1}, we know that the ergodic HJBI equation \eqref{pde1} (resp. \eqref{pde2}) has a Lipschitz viscosity solution $(\rho_1,w_1)$ (resp. $(\rho_2,w_2)$). Then we get from Theorem \ref{th1} that
\begin{equation}\nonumber
\begin{aligned}
\rho_1\geq \sup_{\beta\in\mathcal{B}}\inf_{u\in\mathcal{U}}\limsup_{T\rightarrow\infty}J(T,x,u,\beta(u))(=\rho^+(x))\geq \sup_{\beta\in\mathcal{B}}\inf_{u\in\mathcal{U}}\liminf_{T\rightarrow\infty}J(T,x,u,\beta(u))\geq \rho_1,\\
\rho_2\geq \inf_{\alpha\in\mathcal{A}}\sup_{v\in\mathcal{V}}\limsup_{T\rightarrow\infty}J(T,x,\alpha(v),v)\geq \inf_{\alpha\in\mathcal{A}}\sup_{v\in\mathcal{V}}\liminf_{T\rightarrow\infty}J(T,x,\alpha(v),v)(=\rho^-(x))\geq \rho_2.\\
\end{aligned}
\end{equation}
This implies that $\rho_1=\rho^+(x)$, $\rho_2=\rho^-(x)$, for $x\in\mathbb{R}^n$.
Under the Isaacs condition, the ergodic HJBI equations $(\ref{pde1})$ and
$(\ref{pde2})$ coincide. Then from Corollary \ref{co1}-(i), we get $\rho_1=\rho_2$, which implies the existence of the value for  our SDG with ergodic payoff since $\rho^+(x)=\rho^-(x)$, for $x\in\mathbb{R}^n$.
\end{proof}
The following theorem  is similar to Theorem \ref{th1}, but in the framework of the Abel mean. Its  proof  is similar to that of Theorem \ref{th1}. For the reader's convenience, we give a sketch proof of (iii), the proofs of (i), (ii) and (iv) are analogous.
\begin{theorem}\label{th3}
Suppose that $w$ is Lipschitz on $\mathbb{R}^n$, then we have\\
{\rm (i)} If $(\rho,w)$ is a viscosity subsolution of equation (\ref{pde1}) then
\begin{equation}\label{equ 032701}
\rho\leq \sup_{\beta\in\mathcal{B}}\inf_{u\in\mathcal{U}}\liminf_{\lambda\rightarrow0^+}\lambda E[\int_0^\infty e^{-\lambda s}f(X_s^{x,u,\beta(u)},u_s,\beta(u)_s)ds].
\end{equation}
{\rm (ii)} If $(\rho,w)$ is a viscosity supersolution of equation (\ref{pde1}) then
\begin{equation}\label{equ 032702}
\rho\geq \sup_{\beta\in\mathcal{B}}\inf_{u\in\mathcal{U}}\limsup_{\lambda\rightarrow0^+}\lambda E[\int_0^\infty e^{-\lambda s}f(X_s^{x,u,\beta(u)},u_s,\beta(u)_s)ds].
\end{equation}
{\rm (iii)} If $(\rho,w)$ is a viscosity subsolution of equation (\ref{pde2}) then
\begin{equation}\label{equ 032703}
\rho\leq\inf_{\alpha\in\mathcal{A}}\sup_{v\in\mathcal{V}}\liminf_{\lambda\rightarrow0^+}\lambda E[\int_0^\infty e^{-\lambda s}f(X_s^{x,\alpha(v),v},\alpha(v)_s,v_s)ds].
\end{equation}
{\rm (iv)} If $(\rho,w)$ is a viscosity supersolution of equation (\ref{pde2}) then
\begin{equation}\label{equ 032704}
\rho\geq\inf_{\alpha\in\mathcal{A}}\sup_{v\in\mathcal{V}}\limsup_{\lambda\rightarrow0^+}\lambda E[\int_0^\infty e^{-\lambda s}f(X_s^{x,\alpha(v),v},\alpha(v)_s,v_s)ds].
\end{equation}
\end{theorem}
\begin{proof}
Step 1. We show that \eqref{equ 032703} holds for $w\in C_b^3(\mathbb{R}^n)$.

We notice that \eqref{equ 2.11} and $v^{\theta}\in\mathcal{V}$ constructed in the proof of Theorem \ref{th1} still hold in this case. For any fixed $T>0$, $m=[T^2]$ and $\theta=\frac{T}{m}$, arbitrary but fixed $u\in \mathcal{U}$, applying It\^o's formula to $e^{-\lambda s}w(X_s^{x,u,v^{\theta}})$, similar to \eqref{equ 2.13}, we have
\begin{equation}\nonumber
\begin{aligned}
&E[e^{-\lambda (i+1)\theta}w(X_{(i+1)\theta}^{x,u,v^\theta})-e^{-\lambda i\theta}w(X_{i\theta}^{x,u,v^\theta})]\\
=&E\int_{i\theta}^{(i+1)\theta}e^{-\lambda s}\Big(-\lambda w(X_s^{x,u,v^\theta})+\langle b(X_s^{x,u,v^\theta},u_s,v_s^{\theta}), Dw(X_s^{x,u,v^\theta})\rangle\\
&\qquad\qquad\qquad\qquad +\frac{1}{2}tr\big((\sigma\sigma^*)(X_s^{x,u,v^\theta},u_s,v_s^{\theta})\cdot D^2w(X_s^{x,u,v^\theta})\big)\Big)ds\\
\geq& -\lambda E[\int_{i\theta}^{(i+1)\theta}e^{-\lambda s}w(X_{s}^{x,u,v^\theta})ds]
+\rho\int_{i\theta}^{(i+1)\theta}e^{-\lambda s}ds\\
&-CE[\sup_{i\theta\leq s\leq (i+1)\theta}|X_s^{x,u,v^\theta}-X_{i\theta}^{x,u,v^\theta}|^2]^{\frac 1 2}\int_{i\theta}^{(i+1)\theta}e^{-\lambda s}ds-E[\int_{i\theta}^{(i+1)\theta}e^{-\lambda s}f(X_s^{x,u,v^\theta},u_s,v_s^{\theta})ds]\\
\geq& \big(-\lambda C+\rho-C(\theta+\theta^{\frac 1 2})\big)\int_{i\theta}^{(i+1)\theta}e^{-\lambda s}ds-E[\int_{i\theta}^{(i+1)\theta}e^{-\lambda s}f(X_s^{x,u,v^\theta},u_s,v_s^{\theta})ds].
\end{aligned}
\end{equation}
Taking $\sum\limits_{i=0}^{m-1}$ and multiplying  $\lambda$ on both sides, we get
$$
\lambda E[e^{-\lambda T}w(X^{x,u,v^{\theta}}_T)]-\lambda w(x)\geq \big(-\lambda C+\rho-C(\theta+\theta^{\frac 1 2})\big)(1-e^{-\lambda T})-\lambda E[\int_{0}^{T}e^{-\lambda s}f(X_s^{x,u,v^{\theta}},u_s,v_s^{\theta})ds].$$
Letting $T\rightarrow\infty$ and $\lambda\rightarrow 0$, we have
$$\rho\leq\liminf_{\lambda\rightarrow 0^+} \lambda E[\int_{0}^{\infty}e^{-\lambda s}f(X_s^{x,u,v^{\theta}},u_s,v_s^{\theta})ds].$$
For every $\alpha\in\mathcal{A}$, similar to \eqref{052101} there exist  $\beta^{\theta}\in\mathcal{B}$ and $(\tilde{u}^{\theta},\tilde{v}^{\theta})\in\mathcal{U}\times\mathcal{V}$ such that $\alpha(\tilde{v}^{\theta})=\tilde{u}^{\theta}$, $\beta^{\theta}(\tilde{u}^{\theta})=\tilde{v}^{\theta}$, $dsdP$-a.e., on $[0,\infty)$.
Therefore,
$$\rho\leq \inf_{\alpha\in\mathcal{A}}\sup_{v\in\mathcal{V}}\liminf_{\lambda\rightarrow 0^+}\lambda E[\int_{0}^{\infty}e^{-\lambda s}f(X_s^{x,\alpha(v),v},\alpha(v)_s,v_s)ds].$$
Step 2. The general case.

We construct an approximation sequence $w_\delta^\varepsilon\in C_b^3(\mathbb{R}^n)$ of $w$ in a same manner as defined in \eqref{equ 2.20} and \eqref{equ 2.26}. Then combining Step 1 and \eqref{equ 030801}, similar to \eqref{2022041301} we obtain
\begin{equation}\label{equ 032705}
\begin{aligned}
\rho
\leq&\inf_{\alpha\in\mathcal{A}}\sup_{v\in\mathcal{V}}\liminf_{\lambda\rightarrow 0^+}\lambda E[\int_{0}^{\infty}e^{-\lambda s}\big(f({X_s^{x,\alpha(v),v}},\alpha(v)_s,v_s)
+g_\delta({X_s^{x,\alpha(v),v})\cdot(1+|X_s^{x,\alpha(v),v}|)}\big)ds]
+{k(\varepsilon)}.
\end{aligned}
\end{equation}
Similarly to \eqref{051302}, we obtain that, for all $R>0$,{
\begin{equation}\label{032709}
\begin{aligned}
\lambda E[\int_{0}^{\infty}e^{-\lambda s}g_\delta(X_s^{x,\alpha(v),v})ds]\leq \frac{C}{R}+ Ch(\delta)+C(1-e^{-\lambda})+\frac{C}{r^{\frac n 2}}\frac{1}{\sqrt{R}}+C\sqrt{n}r.
\end{aligned}
\end{equation}}
From \eqref{equ 032705} and \eqref{032709}, we conclude that
\begin{equation}\nonumber
\rho
\leq\inf_{\alpha\in\mathcal{A}}\sup_{v\in\mathcal{V}}\liminf_{\lambda\rightarrow 0^+}\lambda E[\int_{0}^{\infty}e^{-\lambda s}f(X_s^{x,\alpha(v),v},\alpha(v)_s,v_s)ds]+\frac{C}{R}+Ch(\delta)+\frac{C}{r^{\frac n 2}}\frac{1}{\sqrt{R}}+C\sqrt{n}r+k(\varepsilon).
\end{equation}
By first letting $R\rightarrow \infty$, and then sending $r,\delta,\varepsilon\rightarrow 0$ sequentially, we obtain the desired result.
\end{proof}
We observe from the proof of Theorem \ref{th1} that inequality \eqref{equ 2.4} still holds if we change the order of $\sup_{\beta\in\mathcal{B}}\inf_{u\in\mathcal{U}}$  and $\liminf_{T\rightarrow\infty}$. Similar situations also hold for \eqref{equ 2.5}-\eqref{equ 2.7} and for the results of Theorem \ref{th3}.
\begin{remark}\label{re-AT-theorem}
Let $(\rho,w)$ be a viscosity solution of equation (\ref{pde1}). From  Theorem \ref{th1}, Corollary \ref{co1}-\emph{(ii)} and Theorem \ref{th3} we get the following
\begin{equation}\nonumber
\begin{aligned}
&\sup_{\beta\in\mathcal{B}}\inf_{u\in\mathcal{U}}\lim_{T\rightarrow\infty}\frac 1 T E \int_0^Tf(X_s^{x,u,\beta(u)},u_s,\beta(u)_s)ds(=\rho)\\
=&\sup_{\beta\in\mathcal{B}}\inf_{u\in\mathcal{U}}\lim_{\lambda\rightarrow 0^+}\lambda E\int_0^\infty e^{-\lambda s}f(X_s^{x,u,\beta(u)},u_s,\beta(u)_s)ds.
\end{aligned}
\end{equation}
This property can be seen as a generalisation of the
 classical Abelian-Tauberian theorem (see \cite{AB2003, S1979}), stating that
$$\lim_{T\rightarrow\infty}\frac 1 T \int_0^T\varphi(t)dt=\lim_{\lambda\rightarrow 0^+}\lambda \int_0^\infty\varphi(t)e^{-\lambda t}dt,$$
if either one of the two limits exists, to the  SDG cases.
\end{remark}
Finally, we present a dynamic programming principle for the viscosity solution of \eqref{pde1} and \eqref{pde2}.
\begin{theorem}\label{DPP}
Suppose that (H1)-(H2) and (H4) are satisfied. Then, the following statements hold:
\\
{\rm (i)} If $(\rho,w)$ is a viscosity solution of (\ref{pde1}), then for any $T>0$,
\begin{equation}\nonumber
w(x)=\sup_{\beta\in\mathcal{B}}\inf_{u\in\mathcal{U}} E[\int_0^T f(X_s^{x,u,\beta(u)},u_s,\beta(u)_s)ds+w(X_T^{x,u,\beta(u)})]-\rho T.
\end{equation}
{\rm (ii)} If $(\rho,w)$ is a viscosity solution of (\ref{pde2}), then for any $T>0$,
\begin{equation}\nonumber
w(x)=\inf_{\alpha\in\mathcal{A}}\sup_{v\in\mathcal{V}} E[\int_0^T f(X_s^{x,\alpha(v),v},\alpha(v)_s,v_s)ds+w(X_T^{x,\alpha(v),v})]-\rho T.
\end{equation}
\end{theorem}
\begin{proof}
The proof of this Theorem is the same as that of Theorem \ref{th1} without letting $T\rightarrow\infty$, we only give the proof of (ii), since (i) can be proved similarly.

If $(\rho,w)$ is a viscosity solution of (\ref{pde2}), taking the supremum over all controls $v$ and infimum over all strategies $\alpha$ on both sides of  \eqref{equ 2.14} we obtain
\begin{equation}\nonumber
\inf_{\alpha\in\mathcal{A}}\sup_{v\in\mathcal{V}}E[w(X_{T}^{x,\alpha(v),v})
+\int_{0}^{T}f(X_s^{x,\alpha(v),v},\alpha(v)_s,v_s)ds]
\geq w(x)-CT (\theta+\theta^{\frac{1}{2}})+\rho T.
\end{equation}
Letting {$m\rightarrow\infty$ (i.e., $\theta\rightarrow 0+$)}, we get
\begin{equation}\nonumber
\inf_{\alpha\in\mathcal{A}}\sup_{v\in\mathcal{V}}E[w(X_{T}^{x,\alpha(v),v})+\int_{0}^{T}f(X_s^{x,\alpha(v),v},\alpha(v)_s,v_s)ds]
\geq w(x)+\rho T.
\end{equation}
The other inequality can be proved analogously.
\end{proof}
\begin{remark}
It is worth pointing out that  assumption {\rm (H3)} is not required in Theorem \ref{DPP}, while this condition is necessary in Theorem \ref{th1}. This is mainly because  (H3) guarantees the boundedness of the state process over the infinite time horizon $[0,\infty)$ (see Lemma \ref{estimates}), then the term
$\frac{1}{T}E[w(X_T^{x,u,v^\theta})-w(x)]$ in \eqref{equ 2.14} tends to $0$ as $T\rightarrow \infty$. However, for Theorem \ref{DPP} we consider the related DPP for fixed finite $T$. 
\end{remark}
\section{\protect \large Application in pollution accumulation problems with long-run average welfare}
We apply our results to study a type of pollution accumulation and consumption problem, which has been a major topic of dynamic economics and environmental policy making (see \cite{KM2007, KSZ1972, P1972}). We consider a situation where an economy consumes some goods and generates pollutions through its consumption. Suppose that the pollution stock, which gradually reduces, is described by the scalar quantity $X_t=\max\{Y_t,0\}$ at time $t$, where $Y$ is a controlled diffusion process given by
\begin{equation}\label{2020103101}
dY_t=[u_t-v_tY_t]dt+{\sigma(Y_t)} dB_t,\ Y_0=x>0.
\end{equation}
Here,  the processes $u_t$ and $v_t$ represent the rate of consumption and  the  decay rate of pollution at time $t$, respectively, {$\sigma:\mathbb{R}\rightarrow\mathbb{R}^d$ is a bounded Lipschitz function.  Thus, the process $Y$ may be degenerate.}
We comment that the decay rate of the pollution is no longer a constant as considered in \cite{KM2007} since there exists some uncertainty on the values of this decay rate (see \cite{JL2014} for more details). Let us point out that the control  $v$  can be seen as a parameter and interpreted as the natural cleaning of pollution through winds, rains, etc., which is obviously unknown and
unobservable to us.
Assume that the control state spaces are
 $$U=[0,\gamma],\ V=[a,b],$$
where the constant $\gamma$ is usually imposed by worldwide protocols (e.g., Kyoto Protocol) in order to promote sustainable development, $a$ and $b$ are given positive constants.

The associated long-run average social welfare is given by the utility of the net consumption of the disutility of pollution, i.e.
\begin{equation}\label{2020103102}
J(x,u,v)=\liminf_{T\rightarrow\infty}\frac{1}{T}E[\int_0^T\big(g(u_t)-f(X_t)\big)dt],
\end{equation}
where $g\in C^2(0,\infty)$ and $f\in C(0,\infty)$ represent the utility of consumption and the disutility function of pollution, respectively. Moreover, they are assumed to satisfy the following properties:
$$g'>0,\ g''<0,\ g'(\infty)=g(0)=0,\ g'(0+)=g(\infty)=\infty;$$
$$f(x)\geq 0,\ \text{convex\ and\ Lipschitz}.$$
The objective is to find an optimal consumption rate $u^*$ to maximise long-run average social welfare $J(x,u,v)$ under the worst-case scenario (or robust control) $v^*$.

When the decay rate is constant, such a problem has been studied by many authors, among others, Kawaguchi and Morimoto \cite{KM2007}, Nguyen and Yin \cite{NY2016} with a switching diffusion system. In this situation, it can be regarded as a stochastic control problem and the optimal consumption rate can be derived from the related Hamilton-Jacobi-Bellman equation. When the decay rate is no longer constant,
we address this problem via our two-player infinite time horizon SDG model with one controller (controlling $u$) against ``nature" (controlling $v$), where the dynamics is given by \eqref{2020103101}, and the ergodic payoff criterion $J$ has the following equivalent formulation:
 \begin{equation}\nonumber
J(x,u,v)=\liminf_{T\rightarrow\infty}\frac{1}{T}E[\int_0^T\Big(g(u_t)-f(\frac{Y_t+|Y_t|}{2})\Big)dt].
\end{equation}
It is worth pointing out that a similar formulation was also studied in \cite{JL2014}, in which  some additional conditions such as non-degenerate, Lyapunov stability and feedback control conditions were required in order to study the robust optimal consumption rate. The advantage of our approach is that these additional assumptions are not necessary, which makes our approach better applicable.

The related Hamiltonian function $H$ has the following form
 \begin{equation}\nonumber
H(y,p,A,u,v)=[pu+g(u)]-ypv+{\frac 1 2 |\sigma(y)|^2 A}-f(\frac{y+|y|}{2}),\ (y,p,A,u,v)\in \mathbb{R}\times\mathbb{R}\times\mathbb{R}\times U\times V.
\end{equation}
It is easy to verify that the assumptions (H1)-(H3) and the Isaac condition hold. Since $H$ is continuous on compact metric spaces $U$ and $ V$, there exist  mappings $\bar{u}:\mathbb{R}\rightarrow U$ and $\bar{v}:\mathbb{R}\times\mathbb{R}\rightarrow V$ such that
  \begin{equation}\label{2020110203}
\bar{u}(p)=\argsup_{u\in U}[pu+g(u)],\ \bar{v}(y,p)=\arginf_{v\in V}(-ypv).
\end{equation}
 Then applying Theorem \ref{th2}, we get the following results.
\begin{theorem}\label{th4}
Let $(\rho,w)$ be a viscosity solution of the  ergodic HJBI equation
  \begin{equation}\nonumber
\rho=\sup_{u\in U}\inf_{v\in V}H(y,Dw(y),D^2w(y),u,v),\ y\in\mathbb{R}.
\end{equation}
Then $\rho$ is the value for the long-run average social welfare associated with \eqref{2020103101} and \eqref{2020103102}. In addition, the optimal consumption rate  and the worst case decay rate of pollution are given by the following feedback form
  \begin{equation}\nonumber
u^*(y)=\bar{u}(Dw(y)),\ v^*(y)=\bar{v}(y,Dw(y)),
\end{equation}
where the mappings $\bar{u}$ and $\bar{v}$ are given in \eqref{2020110203}.
\end{theorem}
\begin{remark}
{\rm (i)} Since the stock of pollution is nonnegative, thus we consider directly the following constrained equation
\begin{equation}\label{2020110301}
\rho=\sup_{u\in U}\inf_{v\in V}H(x,Dw(x),D^2w(x),u,v),\ x\in\mathbb{R}^+.
\end{equation}
We can still obtain the existence of the viscosity solution by vanishing the limit in the discounted cases, which is similar to Theorem \ref{th3.1}.

{\rm (ii)} If we take $g(u)=2u^{\frac 1 2}$, $f(x)=d\cdot x$ with some constant $d>0$, {we calculate from \eqref{2020110203} that
$$\bar{u}(p)=|\text{Proj}_{[0,\sqrt{\gamma}]}(-\frac{1}{p})|^2I_{\{p<0\}}+\gamma I_{\{p\geq 0\}},\ \bar{v}(y,p)=bI_{\{yp>0\}}+aI_{\{yp\leq 0\}},$$
from which} it is easy to check that
$$\rho=-\frac d a dist^2(\frac a d, [0,\sqrt{\gamma}])+\frac a d,\ w(x)=-\frac d a x,$$
is a classical solution of the ergodic HJBI equation \eqref{2020110301}. In this situation, the optimal consumption rate $u^*=Proj^2_{[0,\sqrt{\gamma}]}(\frac a d)$ and the worst case decay rate of pollution $v^*=a$. As shown in the expression of $\rho$, the long-run average welfare is independent of the upper value of the decay rate $b$. The reason is that  we consider the problem under the worst case scenario, which means that  the decision maker is risk aversion. Moreover, if the lower value of the decay rate $a$ is greater than $d\sqrt{\gamma}$, the long-run average welfare and the optimal consumption rate have the following forms
$$\rho=2\sqrt{\gamma}-\gamma\frac d a,\ u^*=\gamma,$$
which gives the relation between  the long-run average welfare and the robust decay rate. Meanwhile, it shows that the lowest decay rate is $d\sqrt{\gamma}$ if one always wants to pursue the largest possible consumption rate $\gamma$.
\end{remark}

\section{ {\protect \large Appendix: The approximation of viscosity solution}}
We first recall the definition of sup-convolution and inf-convolution. We say that $w_\varepsilon(x)$ and $w^\varepsilon(x)$ is the inf-convolution and sup-convolution of $w$, respectively,  if
$$w_\varepsilon(x)=\inf_{y\in\mathbb{R}^n}\{w(y)+\frac{|x-y|^2}{2\varepsilon}\},\
w^\varepsilon(x)=\sup_{y\in\mathbb{R}^n}\{w(y)-\frac{|x-y|^2}{2\varepsilon}\},\ x\in\mathbb{R}^n.$$
It is well known that sup-convolution and inf-convolution yield good approximations of the viscosity subsolution and supersolution, respectively (see \cite{CIL1992,JLS1988,LL1986}).  We adapt this property to our framework, and the following statement involving two controls is new. Its proof is rather standard, but for the reader's convenience  we give the proof here.

\begin{lemma}\label{Appendix}
{Suppose Assumptions (H1)-(H2) hold, then we have the following statements:} \\
{\rm (1)} If $(\rho,w)$ is a viscosity supersolution for equation \eqref{pde2} {and $w$ is Lipschitz}, then $(\rho,w_\varepsilon)$ is a viscosity supersolution of the following equation
\begin{equation}\label{equ 052301}
\rho=\sup_{v\in V}\inf_{u\in U}H(x,Dw_\varepsilon(x),D^2w_\varepsilon(x),u,v)-{
k(\varepsilon)(1+|x|)},
\end{equation}
where $k(\varepsilon)\rightarrow 0$, as $\varepsilon\rightarrow 0$. \\
{\rm (2)} If $(\rho,w)$ is a viscosity subsolution of the equation \eqref{pde2} {and $w$ is Lipschitz}, then $(\rho,w^\varepsilon)$ is a viscosity subsolution of the following equation
\begin{equation}\label{equ 022801}
\rho=\sup_{v\in V}\inf_{u\in U}H(x,Dw^\varepsilon(x),D^2w^\varepsilon(x),u,v)+
{k(\varepsilon)(1+|x|)},
\end{equation}
where $k(\varepsilon)\rightarrow 0$, as $\varepsilon\rightarrow 0$.
\end{lemma}
\begin{proof}
We only prove (2), the proof of (1) is analogous.

Let $\phi\in C_b^3(\mathbb{R}^n)$ be a test function such that $w^\varepsilon-\phi$ has a local maximum at $x_0\in\mathbb{R}^n$.
We can find {$y_0\in\bar{B}(x_0,2M\varepsilon)$} such that
$$w^\varepsilon(x_0)=w(y_0)-\frac{|x_0-y_0|^2}{2\varepsilon},$$
where $M$ is the Lipschitz constant of $w$.
Define $\psi(y)=\phi(x_0-y_0+y)+\frac{|x_0-y_0|^2}{2\varepsilon}$, $y\in\mathbb{R}^n$. Obviously, $\psi\in C_b^3(\mathbb{R}^n)$, and {$w-\psi$ attains its local
maximum at $y_0$}. Since $(\rho,w)$ is a viscosity subsolution of the ergodic HJBI equation (\ref{pde2}) and $D\psi(y_0)=D\phi(x_0)$, $D^2\psi(y_0)=D^2\phi(x_0)$, we get
\begin{equation}\label{equ 2.22}
\begin{aligned}
\rho\leq& \sup_{v\in V}\inf_{u\in U}\{\frac{1}{2}tr\big(\sigma\sigma^*(y_0,u,v)\cdot D^2\psi(y_0)\big)+b(y_0,u,v)\cdot D\psi(y_0)+f(y_0,u,v)\}\\
=& \sup_{v\in V}\inf_{u\in U}\{\frac{1}{2}tr\big(\sigma\sigma^*(y_0,u,v)\cdot D^2\phi(x_0)\big)+b(y_0,u,v)\cdot D\phi(x_0)+f(y_0,u,v)\}.
\end{aligned}
\end{equation}
Using Assumptions (H1)-(H2) and (\ref{equ 2.22}), we have
\begin{equation}\label{equ 2.23}
\rho\leq \sup_{v\in V}\inf_{u\in U}\{\frac{1}{2}tr\big(\sigma\sigma^*(x_0,u,v)\cdot D^2\phi(x_0)\big)+b(x_0,u,v)\cdot D\phi(x_0)+f(x_0,u,v)\}+{k(\varepsilon)(1+|x_0|)},
\end{equation}
where $k(\varepsilon)= C\varepsilon$. Thus, we find that $(\rho, w^\varepsilon)$ is a viscosity subsolution of (\ref{equ 022801}).
\end{proof}
{We remark that the term $k(\varepsilon)(1+|x|)$ in both \eqref{equ 052301} and \eqref{equ 022801} will be changed to $k(\varepsilon)$ if the additional assumption (H4) holds. Indeed, one can easily verify it from \eqref{equ 2.23} combined with assumption (H4).}





 \end{document}